\documentclass[reqno,11pt,centertags]{article}
\usepackage{amsmath,amsthm,amscd,amssymb,latexsym,upref}
\date{\today}
\usepackage[T2A,OT1]{fontenc}
\usepackage[ot2enc]{inputenc}
\usepackage[russian,english]{babel}
\usepackage[margin=3.5 cm]{geometry}
\newcommand{\bbD}{{\mathbb{D}}}
\newcommand{\bbE}{{\mathbb{E}}}

\newcommand{\bbR}{{\mathbb{R}}}

\newcommand{\bbZ}{{\mathbb{Z}}}
\newcommand{\bbC}{{\mathbb{C}}}

\newcommand{\bbT}{{\mathbb{T}}}

\newcommand{\cF}{{\mathcal{F}}}

\newcommand{\cH}{{\mathcal{H}}}

\newcommand{\fF}{{\mathfrak{F}}}

\newcommand{\ff}{{\mathfrak{f}}}

\newcommand{\e}{{\epsilon}}

\newcommand{\z}{\zeta}

\renewcommand{\Re}{\text{\rm Re}\,}
\renewcommand{\Im}{\text{\rm Im}\,}

\allowdisplaybreaks \numberwithin{equation}{section}
\newtheorem{theorem}{Theorem}[section]

\newtheorem{lemma}[theorem]{Lemma}
\newtheorem{proposition}[theorem]{Proposition}

\newtheorem{corollary}[theorem]{Corollary}

\theoremstyle{definition}
\newtheorem{definition}[theorem]{Definition}

\newtheorem{remark}[theorem]{Remark}

\def\be{\begin{equation}}
\def\ee{\end{equation}}
\def\bea{\begin{eqnarray}}
\def\eea{\end{eqnarray}}
\def\bean{\begin{eqnarray*}}
\def\eean{\end{eqnarray*}}

\def\restr#1{\,\vrule\,\lower1ex\hbox{$#1$}}

\def\a{\alpha}
\def\b{\beta}

\def\e{\epsilon}

\def\g{\gamma}
\def\G{\Gamma}

\def\l{\lambda}
\def\L{\Lambda}
\def\o{\omega}
\def\O{\Omega}

\def\t{\theta}
\def\z{\zeta}

\title
{Martin Functions of Fuchsian Groups and \\ Character Automorphic Subspaces of the Hardy Space in the Upper Half Plane}
\author{A. Kheifets and P. Yuditskii\thanks{Supported by the Austrian Science Fund FWF, project no: P29363-N32.}}

\begin{document}

\maketitle

\begin{center}
\textit{Dedicated to our teacher Prof. V. E. Katsnelson
\footnote{Following the Russian tradition, we use the initials V.E. for Victor Emmanuilovich. V.E. was the most unusual and therefore the most attractive person among professors in the math department for students of our generation. The method he used to bring us in mathematics was also very much unusual for our time and our country: V.E. took us (four first year Master students) to a REAL mathematical conference. Actually, it was a school, but definitely of the highest conference level. At this school, we were learning the J-theory for two weeks at least 6 hours a day. V.E. was one of the lecturers, highly enthusiastic. We studied the theory together with the most prominent professors of our department. At the lunchtime, during a ski trip, or at the night lectures we were able to meet the authors of practically all popular textbooks of that time. V.E. was our guide to this new world. If indeed a personality is completely determined by the first 3 years of our life, our mathematical personalities definitely were determined by these first two weeks of our mathematical childhood.}
on the occasion of his 75-th birthday}
\end{center}

\begin{abstract}
We establish exact conditions
for non triviality of all subspaces of the standard Hardy space
 in the upper half plane, that consist of the character automorphic functions with respect to the action of a discrete subgroup of $SL_2(\bbR)$. Such spaces are the natural objects in the context of the spectral theory of almost periodic differential operators and in the asymptotics of the approximations by  entire functions.
A naive idea: it should be completely parallel to the celebrated Widom characterization for Hardy spaces on Riemann surfaces with a minor modification, namely, one has to substitute the Green function of the domain with the Martin function. Basically, this is correct, but...

\end{abstract}
\section{Introduction}

Harold Widom discovered that the asymptotic behaviour of orthogonal polynomials associated with a system of curves in the complex plane can be expressed in terms of the \textit{reproducing kernels} of the Hardy spaces of \textit{character automorphic} functions on the complementary domain (containing infinity component) \cite{Widom}. Later on in \cite{Widom71} he found a condition that guaranties  non triviality of all these spaces on infinitely connected domains. Essentially this created a foundation for the most comprehensive currently available function theory on multiply connected domains (and Riemann surfaces) \cite{Hasu}.

In its turn the theory of character automorphic Hardy spaces appeared to be the most efficient tool in solving inverse spectral and scattering problems for ergodic, almost periodic difference/differential operators and for their perturbations, see e.g. \cite{SY,  VYis}, see also \cite{You}. We mention here that a certain reverse influence also
took place, see \cite{VY}.

Another broad field of research to be naturally mentioned here is the spectral theory of commuting non-self adjoint operators and the interpolation theory on the Riemann surfaces, see e.g. \cite{LKMV, BV, AV}.

Viewing the Hardy spaces on the Riemann surfaces in terms of the \textit{universal covering} is extremely convenient for analysts, see e.g. \cite{PF}. Under this approach we realize the corresponding Hilbert space on the Riemann surface as a subspace of the standard $H^2$ in the disc consisting of the functions automorphic (character automorphic with a prescribed character) with respect to the action of a certain Fuchsian group $\G$.
Effectively, an essential part of the book \cite{Hasu} can be substituted with a single paper \cite{Pom} by Christian Pommerenke if one uses this approach.

In this paper we give precise conditions for non triviality of all subspaces of the standard Hardy space
$\cH^2=H^2_{\bbC_+}$ on the upper half plane, see e.g. \cite[Lecture XI]{Nik}, that consist of the character automorphic functions with respect to the action of a discrete subgroup $\G$ of $SL_2(\bbR)$. Such spaces are natural in the context of the spectral theory of differential operators and in the asymptotics of approximations by  entire functions.

A naive idea: it should be completely parallel to the corresponding Widom characterization with a minor modification, namely, one has to substitute the Green function of the domain with the Martin function. Basically, this is correct ..., but, as we will see, one more condition (see condition (B) in our main Theorem \ref{mth-15oct01}) should be surprisingly added to Widom type condition (A) in this case.
Moreover, in Remark \ref{re110} we will show why the approach of \cite{Pom} cannot work here in principle without essential modification.

Going to the precise statement we will introduce some notations, recall definitions and some facts.
We restrict ourself to \textit{Denjoy domains} (complements to real closed sets), which are \textit{regular} in the sense of the potential theory.

Let $E=\bbR\setminus \cup_{j\in\bbZ}(a_j,b_j)$ be a closed subset of $\bbR$, $E\not=\bbR$, unbounded in the following sense
\begin{equation}\label{10sep0}
\forall N>0\ \exists \l_\pm\in E: \l_+>N\ \text{and}\  \l_-<-N.
\end{equation}
Regularity of $E$ means that there exists a positive harmonic function in $\O=\bbC\setminus E$ with the only logarithmic singularity at a point $\l_0\in\O$ that is continuous up to the boundary of $\Omega$ and vanishes there. This function is called the Green function and is denoted by $G(\l,\l_0)$.

One can give a parametric description of regular Denjoy domains in terms of the special conformal mappings that were introduced
by Akhiezer and Levin, see \cite{Lev}, and that are extensively used in the spectral theory, see \cite{Mar}; for a modern point of view see \cite{EYu}, in particular for the proofs of Theorem \ref{th11}, Propositions \ref{prop12} and \ref{prop17}.
%We will fix an interval $(a_0, b_0)$, a component of $\bbR\setminus E$, and a normalization point $\lambda_*$ in it.
Let
\begin{equation}\label{10sep3}
\Pi=\{\xi+i\eta, \eta>0, \xi\in(0,\pi)\}\setminus \cup_{j\not=0}\{\o_{j}+i\eta: \eta\in(0,h_{j})\},
\end{equation}
where
$\{(\o_{k},h_{k})\}_{k\not=0}$ is any collection of numbers such that
\begin{equation}\label{10sep1}
\o_{k}\in(0,\pi),\quad \o_{k}\not=\o_{j}\quad \text{for}\quad k\not=j,
\end{equation}
and
\begin{equation}\label{10sep2}
h_{k}>0,\quad \lim_{k\to\infty} h_{k}=0.
\end{equation}
Domains $\Pi$ of this type are called
the \textit{regular combs}.
\begin{theorem}\label{th11} Let $(a_0, b_0)\subset\bbR$ and $\l_*\in (a_0, b_0)$. Let
$\Pi$ be an arbitrary comb of the form \eqref{10sep3} - \eqref{10sep2}
with parameters
$$
(\o_{k},h_{k}),\quad{k\not=0},
$$
and let
$$
\t_{\l_*}:\bbC_+\to \Pi
$$
be  the conformal mapping of $\bbC_+$ onto $\Pi$,
normalized as follows
\be\label{05dec01}
\t_{\l_*}(\l_*)=\infty,\ \t_{\l_*}(b_0)=0,\ \t_{\l_*}(a_0)=\pi.
\ee
Then $\t_{\l_*}$
can be extended by continuity to the real axis  and the set $E:=\t_{\l_*}^{-1}([0,\pi])$ is regular.  Moreover, $\Im \t_{\l_*}(\l)$ can be extended to the domain $\O:=\bbC\setminus E$ as a single-valued function, and for this extension we have
\begin{equation}\label{10sep5}
\Im\t_{\l_*}(\l)=G(\l,\l_*),
\end{equation}
where $G(\l, \l_*)$ is the Green function of $\O$. Due to normalization \eqref{05dec01},
$\t_{\l_*}(\infty)\in(0,\pi)$. If it does not coincide with the base point of a slit, i.e., $\t_{\l_*}(\infty)\not=\o_j$,
$j\in\bbZ$, then the set $E$ has property \eqref{10sep0}.

Conversely, let $E$ be a regular set, let $(a_0,b_0)$ be a component of $\bbR\setminus E$ and let $\l_*\in(a_0,b_0)$. Then there exists a comb $\Pi_{\l_*}$ of the form \eqref{10sep3} - \eqref{10sep2}
with parameters
$$
(\o_{\l_*,k},h_{\l_*,k}),\quad{k\not=0},
$$
such that $E$ corresponds to the base $[0,\pi]$ for the conformal mapping
\footnote{Here $\mathbb C_+$ is considered as a subset of $\mathbb C\setminus E$.}
$
\t_{\l_*}:\bbC_+\to \Pi_{\l_*},
$
normalized as in \eqref{05dec01}. Moreover,
\eqref{10sep5} holds. If $E$ has property \eqref{10sep0}, then $\t_{\l_*}(\infty)$ does not coincide with the base point of a slit, i.e., $\t_{\l_*}(\infty)\not=\o_j$,
$j\in\bbZ$.
\end{theorem}

The function $\t_{\l_*}(\l)$  admits a Schwarz-Christoffel type representation (an infinite analogue of the conformal mapping onto a polygon).

\begin{proposition}\label{prop12}
Assume that $E$ is regular and that $\t_{\l_*}(\l)$ is the conformal mapping on the corresponding comb domain.
Let
\begin{equation*}%\label{10sep6}
\mu_{\l_*,k}:=\t_{\l_*}^{-1}(\o_{\l_*,k}+ih_{\l_*,k})\in (a_k,b_k).
\end{equation*}
Then for $\l\in\O$
\begin{equation}\label{6aug1}
\t'_{\l_*}(\l)=\frac i{\l_*-\l}\frac{\sqrt{(\l_*-a_0)(\l_*-b_0)}}{\sqrt{(\l-a_0)(\l-b_0)}}\prod_{k\not=0}\frac{\l-\mu_{\l_*,k}}{\l_*-\mu_{\l_*,k}}\frac{\sqrt{(\l_*-a_k)(\l_*-b_k)}}{\sqrt{(\l-a_k)(\l-b_k)}}.
\end{equation}
In particular,
%$\Im \t'_{\l_*}(\l)/i>0$ for $\Im\l >0$ and
$\{\mu_{\l_*,k}\}_{k\not=0}$ is the complete list of the critical points of the function $G(\l,\l_*)$, that is, the points where $\nabla G(\l,\l_*)=0$.
\end{proposition}

\begin{definition}
A regular domain is said to be of the Widom type if
\begin{equation}\label{10sep7}
\sum_{\mu:\nabla G(\mu,\l_*)=0}G(\mu,\l_*)
=\sum_{k\ne 0} G(\mu_{\l_*, k}, \l_*)
<\infty.
\end{equation}
\end{definition}
\noindent
Note that, by \eqref{10sep5}, $G(\mu_{\l_*,k},\l_*)=h_{\l_*,k}$ and \eqref{10sep7} is the same as
\begin{equation}\label{10sep8}
\sum_{k\not=0} h_{*k}<\infty.
\end{equation}
Thus, all Denjoy domains of the Widom type are represented by the conformal mappings on the comb domains, where
\eqref{10sep2} should be substituted with a stronger condition \eqref{10sep8}.

According to the \textit{uniformization theorem} there exists an analytic function $\L(z)$ on the upper half plane $\bbC_+$ that sets a one to one correspondence between the domain $\O$ and the factor of $\bbC_+$ under the action of a discrete group $\G\subset SL_2(\bbR)$, that is,
$\L(z)\in\O,$ and
for every $\l\in\O$ there exists $z\in\bbC_+$ such that $\L(z)=\l$. Moreover,
$$
\L(\g(z))=\L(z),\ \text{where}\ \g(z)=\frac{\g^{11}z+\g^{12}}{\g^{21}z+\g^{22}}\ \text{for all}\  \g=
\begin{bmatrix}
\g^{11}&\g^{12}\\
\g^{21}&\g^{22}
\end{bmatrix}\in\G,
$$
and $\L(z_1)=\L(z_2)$ implies that there exists $\g\in\G$ such that
$
z_1=\g(z_2).
$

In terms of the universal covering the Green function $G(\l,\l_*)$ admits the following representation.
Let us fix $z_*$ such that $\Lambda(z_*)=\l_*$.
Consider\footnote{If $E$ is regular, then $\Gamma$ is of the convergent type. That is, the orbit of every point in $\mathbb C_+$ satisfies the Blaschke condition.}
%\marginpar{Obozbnacheniya: $c_\g$  a eshe est' $c_k$ v dvukh smyslakh (1.1) i Lemma 6.1}
\begin{equation}\label{11sep11}
g(z, z_*)=\prod\limits_{\gamma\in\Gamma}\frac{z-\gamma(z_*)}{z-\overline{\gamma(z_*)}}\ C_\gamma,\quad z\in\mathbb C_+,
\end{equation}
where $C_{1_\G}=1$ and for all $\gamma\ne 1_{\G}$
$$
C_\gamma =
\left|\frac{z_*-\gamma(z_*)}{z_*-\overline{\gamma(z_*)}}\right|
\frac{z_*-\overline{\gamma(z_*)}}{z_*-\gamma(z_*)}\ .
$$
 Then
 \begin{equation}\label{11sep2}
G(\L(z), \L(z_*))=-\ln|g(z, z_*)|.
\end{equation}
For this reason $g(z,z_*)$ is called the (complex) Green function of the group $\G$, see \cite{Pom}.
Combining \eqref{10sep5} and \eqref{11sep2}, we get
 \begin{equation}\label{27dec01}
\t_{\L(z_*)} ( \L(z))=-i\ln g(z, z_*).
\end{equation}
Therefore,
 \begin{equation*}%\label{27dec01-der}
\t_{\L(z_*)}' ( \L(z))\cdot\L'(z)=-i\frac{g'(z, z_*)}{g(z, z_*)}.
\end{equation*}
From here we see that the critical points
$\{\mu_{\l_*,k}\}_{k\not=0}$ of the Green function $G(\l,\l_*)$ are images of the zeros of $g'(z, z_*)$ under $\L$.
Then Widom condition \eqref{10sep7} can be written in terms of $g(z, z_*)$ as follows
\be\label{27dec02}
\prod_{k\ne 0}|g(c_{z_*, k}, z_*)|>0,
\ee
where $c_{z_*, k}$ are zeros of $g'(z, z_*)$ in the fundamental domain of $\G$, which is the Blaschke condition on {\it all} the zeros of $g'(z, z_*)$ in the upper half plane.

By $\G^*$ we denote the group of the unimodular characters of $\G$, that is, the functions
$$
\a:\G\to \bbT
\quad
\text{such that}\quad
\a(\g_1\g_2)=\a(\g_1)\a(\g_2), \g_j\in\G.
$$
Note $g(z,z_*)$ is an example of the character automorphic function, that is, there exists $\b_*=\b_{g(\cdot,z_*)}\in\G^*$ such that
$$
g(\g(z),z_*)=\b_*(\g)g(z,z_*),
$$
respectively,
$$
|g(\g(z),z_*)|=|g(z,z_*)|.
$$

Passing by a linear fractional transformation from the unit disk $\bbD$ to the upper half plane $\bbC_+$,
we introduce the classical Hardy space $H^2$ of holomorphic functions on $\bbC_+$, with the norm
\begin{equation}\label{24oct1}
\|f\|^2=\|f\|^2_{H^2}=\int_{\bbR}|f(x)|^2 dm(x), \quad dm(x)=\frac{dx}{1+x^2}.
\end{equation}
\begin{definition}
For a fixed character $\a\in\G^*$ we define
\begin{equation*}%\label{11sep1}
H^2(\a)=\{f\in H^2:\ f(\g(z))=\a(\g)f(z), \forall \g\in\G\}.
\end{equation*}
\end{definition}
The following statement is the Pommerenke version \cite{Pom} of the Widom theorem (recall, in this paper we discuss only Denjoy domains).
\begin{theorem}\label{TP}
Let $\O=\bbC\setminus E$ be a regular Denjoy domain and $\L:\bbC_+/\G\simeq \O$ be its uniformization.
The following conditions are equivalent
\begin{itemize}
\item[(i)] For every $\a\in\G^*$ the space $H^2(\a)$ contains a non constant function.
\item[(ii)] The derivative $g'(z,z_*)$ of the Green function  is a function of bounded characteristic in $\bbC_+$ (a ratio of two bounded holomorphic functions)
\item[(iii)] Widom condition \eqref{10sep7} $($equivalently \eqref{27dec02}$)$ holds.
\end{itemize}
\end{theorem}
Let now $\cH^2$ be the standard Hardy space in the upper half plane, that is,
Smirnov class functions $f$ with finite norm
$$
\|f\|^2=\|f\|^2_{\cH^2}=\int_{\bbR}|f(x)|^2 dx.
$$
\begin{definition}
For a fixed character $\a\in\G^*$ we introduce
\begin{equation*}%\label{11sep1-st}
\cH^2(\a)=\{f\in \cH^2:\ f(\g(z))=\a(\g)f(z), \forall \g\in\G\}.
\end{equation*}
\end{definition}

We express a similar property of non triviality of all $\cH^2(\a)$ spaces  in terms of the Martin function $M(\l)$ in $\O$ (associated to the infinity).

To be more precise, by $M(\l)$, $\l\in\O$, we denote the  symmetric Martin function with respect to the infinity, see e.g. \cite{Lev, EYu, BS}, and the references therein. That is,
$M(\l)$ is a positive harmonic in $\O$ function, continuous up to the boundary with the only exception at the infinity and vanishing at every finite point of the boundary. Symmetry means that
$$
M(\overline\lambda)=M(\lambda).
$$
Such a function is unique up to a positive constant factor.
It also admits a Schwarz-Christoffel type representation.
\begin{proposition}\label{prop17}
Assume that $\O$ is a regular Danjoy domain. All critical points $\mu_k$ of the symmetric Martin function are real, moreover
it has exactly one critical point in each gap
\begin{equation*}%\label{11sep6}
\mu_{k}\in (a_k,b_k), \quad k\in\bbZ.
\end{equation*}
Then for $\l\in\bbC_+$ and a fixed normalization point  $\l_*\in(a_0,\mu_0)$
\begin{equation}\label{11sep7}
\t(\l):=i(\partial_x M)(\l_*)\int_{a_0}^\l\prod_{k\in\bbZ}\frac{\xi-\mu_{k}}{\l_*-\mu_{k}}\frac{\sqrt{(\l_*-a_k)(\l_*-b_k)}}{\sqrt{(\xi-a_k)(\xi-b_k)}}d\xi
\end{equation}
and $M(\l)=\Im\t(\l)$.
\end{proposition}

Note that $\t(\l)$ also generates a conformal mapping of the upper half plane on a special comb domain \cite{EYu}. It can be extended to $\O$ as an (additive) character automorphic function. This can be described in terms of the uniformization: let $m(z)=\t(\L(z))$, then
\be\label{27dec06}
M(\L(z))=\Im m(z),\quad m(\g(z))=m(z)+\eta(\g),
\ee
where $\eta(\g)\in\bbR$, $\eta(\g_1\g_2)=\eta(\g_1)+\eta(\g_2).$
Similar to $g(z,z_*)$, the function $m(z)$ can be called the \textit{(symmetric) complex Martin function of the group} $\G$.

\medskip

Now we can state our main result.
\begin{theorem}\label{mth-15oct01}
Let $\O=\bbC\setminus E$ be a regular Denjoy domain and $\L:\bbC_+/\G\simeq \O$ be its uniformization.
The following conditions are equivalent
\begin{itemize}
\item[(i)] For every $\a\in\G^*$ the space $\cH^2(\a)\not=\{0\}$.
\item[(ii)]
\begin{itemize}
\item[($a$)]  The derivative $m'(z)$ of the  Martin function of the group $\G$ is a function of bounded characteristic;
\item[($b$)]  The Riesz-Herglotz measure correspondent to $m(z)$ is a pure point one.
\end{itemize}
\item[(iii)] The symmetric Martin function $M(\l)$ of the domain $\O$ possesses the following two properties
\begin{itemize}
\item [$(A)$] $\sum\limits_{j\in \bbZ} G(\mu_j,\l_*)<\infty$, where $\mu_j$ are the critical points of the {\em Martin function};
\item [$(B)$] $\lim\limits_{\eta\to+\infty}{M(i\eta)}/\eta>0$.
\end{itemize}
\end{itemize}
\end{theorem}

\begin{remark}
First of all we note that  in our theorem condition ($a$), as an expected counterpart of (ii) in the Pommerenke theorem, should be accompanied by the second condition  ($b$),
respectively, the Widom type condition (A) in our case is accompanied by condition (B) that characterizes a special behaviour of the Martin function at infinity.
\end{remark}
\begin{remark}
(B) is the well known  Akhiezer-Levin condition,  see e.g. \cite{BS}. As soon as (B) holds $M(\l)$ is also called the
\textit{Phragm\'en-Lindel\"of function}, see \cite{Koo} and especially  Theorem on p. 407 in this book.
Condition ($b$) was discussed in \cite{VYM}, see especially  Theorem 5 and Lemma 1 there. It can be equivalently stated in a form similar
to condition (B)
\begin{equation}\label{27dec04}
(b_1)\quad\lim\limits_{y\to+\infty}{M(\L(iy))}/y>0.
\end{equation}
We point out that in condition (B) $i\eta$ belongs to the upper half plane of the domain $\mathbb C\setminus E$, whereas in condition ($b_1$) $iy$ is in the universal cover. It appears that the equivalence of the Akhiezer-Levin condition (B)
and property ($b$), proved in Section~\ref{Akhiezer-Levin} below, is a new result.
\end{remark}

\begin{remark}\label{re110}
Pommerenke's proof of implication (iii) to (ii) in Theorem \ref{TP} is based on an exhaustion of the given domain $\Omega$ by subdomains $\O_{\e}$: connected components of the set
$$
\{\l: G(\l,\l_*)>\e\},
$$
containing $\l_*$. It is highly important in the proof that such domains are finitely connected.  To follow this line in our  proof and to keep under control  the critical points of the Martin function one has to make a similar exhaustion generated by the sets
$$
\{\l: M(\l)>\e\}.
$$
But the simplest example
$$
\O:=\left\{\l: |\cos\l|>\frac 1 2\right\}
$$
shows that the corresponding domains $\O_\e$ remain possibly infinitely connected  for \textit{all} sufficiently small $\e$. Thus, another kind of approximation of the given domain is needed, respectively the proof should be essentially reorganised.
\end{remark}

In this paper we choose the approximation of the group $\G$ by its finitely generated subgroups. The corresponding construction is discussed in Section~\ref{FZ}.
In Section~\ref{PT} we partially reprove Pommerenke Theorem~\ref{TP} (equivalence of (ii) and (iii)) using this approach.
In this part,
it is an essential simplification  of his original construction. Note, though, that we are restricted in our setting to  Denjoy domains only, while Pommerenke's proof is valid for arbitrary Riemann surfaces. Subsection \ref{DM} describes the Martin functions $m(z)$ that possess property
($b$) (equivalently ($b_1$) of \eqref{27dec04}, by Proposition \ref{27dec05}). In Section \ref{Akhiezer-Levin} we prove that condition ($b_1$) and Akhiezer-Levin condition
(B) are equivalent. Finally, in Section \ref{pro} we prove our main Theorem~\ref{mth-15oct01}. The proof is broken into several steps, each one corresponds to a certain implication between assertions (i)--(iii). For the reader's convenience in the Appendix we give proofs of the Carath\'eodory and Frostman theorems, that were essential components of the original Pommernke's proof \cite{Pom}  (given there as references).

\section{Preliminaries}\label{FZ}

The Blaschke condition on a set $\{z_k\}$ for the upper half plane can be written as
\begin{equation}\label{16sep01}
\sum\limits_{k}\frac{\Im z_k}{|z-\overline{z_k}|^2}<\infty,\quad
\Im z_k>0,
\end{equation}
where $z$ is an arbitrary fixed point in the upper half plane.
The convergence in \eqref{16sep01} is uniform in $z$ on compact subsets of the open upper half plane, since
$$
\frac{|z-w|}{|\widetilde z-w|}
$$
is continuous and, therefore, is bounded when $z$ and $\widetilde z$ are in a compact subset of the open upper half plane and $w$ is in the closed lower half plane (including infinity). Hence, the corresponding Blaschke product
$$
\prod\limits_{k}\frac{z-z_k}{z-\overline {z_k}}C_k,
$$
converges uniformly on the compact subsets of $\mathbb C_+$,
where constants $C_k$ are chosen to make the factors positive at one point of the upper half plane.

Since $\Gamma$ is of convergent type, the Blaschke condition holds for the orbit of an arbitrary point $z_*$ in the upper half plane
\begin{equation}\label{15sep01}
\sum\limits_{\gamma\in\Gamma}\frac{\Im \gamma (z_*)}{|z-\overline{\gamma(z_*)}|^2}<\infty,\quad
\Im z>0.
\end{equation}
Hence, $g(z, z_*)$ is well defined by this formula %\eqref{11sep11}
\begin{equation}\label{16sep02}
g(z, z_*)=\prod\limits_{\gamma\in\Gamma}\frac{z-\gamma(z_*)}{z-\overline{\gamma(z_*)}}\ C_\gamma,\quad z\in\mathbb C_+,
\end{equation}
and the convergence is uniform on the compact subsets of $\mathbb C_+$.
Equivalently, $g(z, z_*)$ can be defined as
\begin{equation}\label{23sep01}
g(z, z_*)=\prod\limits_{\gamma\in\Gamma}\frac{\gamma (z)-z_*}{\gamma(z)-\overline{z_*}}\ \widetilde C_\gamma,\quad z\in\mathbb C_+ .
\end{equation}
For the logarithmic derivative of $g(z, z_*)$ we get
\begin{equation}\label{16sep03}
\frac{g'(z, z_*)}{g(z, z_*)}=(z_*-\overline{z_*})
\sum\limits_{\gamma\in\Gamma}
\frac{\gamma'(z)}{(\gamma(z)-z_*)(\gamma(z)-\overline{z_*})}
,\quad z\in\mathbb C_+.
\end{equation}
From here we see that
\begin{equation}\label{16sep04}
g'(z, z_*)=
(z_*-\overline{z_*})
\sum\limits_{\gamma\in\Gamma}
\frac{g(z, z_*)\gamma'(z)}{(\gamma(z)-z_*)(\gamma(z)-\overline{z_*})}
,\quad z\in\mathbb C_+.
\end{equation}
The convergence in \eqref{16sep04} is absolute and uniform on compact subsets of $\mathbb C_+$ due to the uniform convergence in \eqref{15sep01}, see also \eqref{16sep02}, \eqref{23sep01}.

We consider domain $\mathcal F$ that is obtained from the universal covering space $\mathbb C_+$ by removing countably (or finitely) many semi-disks with real centers. We choose one of them to be of radius $1$ with center at $0$ and we label it with index $0$. The universal covering map carries $\mathcal F$ conformally onto the upper half plane in $\mathbb C\setminus E$. The semi-circles are mapped onto the gaps, the real part of the boundary of
$\mathcal F$ is mapped onto $E$. The fundamental domain of the group $\Gamma$ can be obtained by taking the union of $\mathcal F$ with its reflection about the $0$-th semi-circle. We also mention here that generators of the group $\Gamma$ are the compositions of this reflection with  the reflections about the other boundary semi-circles of $\mathcal F$.

We consider domain $\mathcal F_n$ that is obtained from $\mathcal F$ by keeping a finite number of the semi-circles and
replacing the others with their diameters on the real line. We have that
$$
\mathcal F=\bigcap_n\mathcal F_n.
$$
Group $\Gamma_n$ is generated by the compositions of pairs of the reflections about the boundary semi-circles of $\mathcal F_n$.
$\Gamma_n$ is a subgroup of $\Gamma$ and
$$
\Gamma=\bigcup_n\Gamma_n.
$$
We consider the complex Green function for $\Gamma_n$ similar to the one for $\Gamma$ with the same $z_*$
\begin{equation*}%\label{28aug10}
g_n(z, z_*)=\prod\limits_{\gamma\in\Gamma_n}\frac{\gamma (z)-z_*}{\gamma(z)-\overline{z_*}}\ \widetilde C_\gamma,\quad z\in\mathbb C_+.
\end{equation*}
$g_n$ is a divisor of $g$. Therefore,
\be\label{27dec03}
|g_n(z)|\ge|g(z)|,\quad z\in\mathbb C_+ .
\ee
%% and
%% $$
%% G_n(z, z_*)=-\ln|g_n(z, z_*)|\le -\ln|g(z, z_*)|=G(z, z_*).
%% $$
We also mention here that
$$
\frac{ g'_n(z, z_*)}{g_n(z, z_*)}=
(z_*-\overline{z_*})
\sum\limits_{\gamma\in\Gamma_n}
\frac{\gamma'(z)}{(\gamma(z)-z_*)(\gamma(z)-\overline{z_*})}
,\quad z\in\mathbb C_+,
$$
and
\begin{equation}\label{28aug17}
g'_n(z, z_*)=
(z_*-\overline{z_*})
\sum\limits_{\gamma\in\Gamma_n}
\frac{ g_n(z, z_*)\gamma'(z)}{(\gamma(z)-z_*)(\gamma(z)-\overline{z_*})}
.
\end{equation}
Again, the convergence is absolute and uniform on the compact subsets of $\mathbb C_+$, since this is true even for the whole group $\Gamma$ (see \eqref{16sep04}).
\begin{lemma}\label{L:28aug01}
As $n$ goes to $\infty$, $g_n(z, z_*)$ converges to $g(z, z_*)$ uniformly on the compact subsets in $\mathbb C_+$ and
$g'_n(z, z_*)$ converges to $g'(z, z_*)$ uniformly on the compact subsets in $\mathbb C_+$. Let $c_{z_*,k}^{(n)}$ be the
zero of $g'_n(z, z_*)$ on the $k$-th semicircle and
$c_{z_*,k}$ be the zero of $g'(z, z_*)$ on the $k$-th semicircle. Then
$$
c_{z_*,k}^{(n)}\to c_{z_*,k}
$$
for every $k\ne 0$. Moreover, $g_n(c_{z_*,k}^{(n)}, z_*)$ converges to $g(c_{z_*,k}, z_*)$ for every
$k\ne 0$.
\end{lemma}
\begin{proof}
The uniform convergence of $g_n(z, z_*)$ follows from the convergent type of $\Gamma$ (see \eqref{15sep01}, \eqref{16sep02}). The uniform convergence
of $g'_n(z, z_*)$ follows from the uniform convergence of $g_n(z, z_*)$ (by local Cauchy integral formula). The convergence of $c_{z_*,k}^{(n)}$ follows from the {\em regularity of $E$} and from the uniform convergence of $g'_n(z, z_*)$, by the Rouche's Theorem. The last assertion is obtained by combining the uniform convergence of $g_n(z, z_*)$ with the convergence of $c_{z_*,k}^{(n)}$.
\end{proof}

\section{Pommerenke Theorem}\label{PT}
\begin{theorem}\label{Pommerenke}
Let $c_{z_*,k}$ be the zeros of $g'(z, z_*)$, one on each semicircle on the boundary of $\cF$, except for the $0$-th one.
Assume that they satisfy the Widom condition \eqref{27dec02}
\begin{equation}\label{28aug01}
\prod_{k\ne 0}|g(c_{z_*, k}, z_*)|>0.
\end{equation}
Then $g'(z, z_*)$ is of bounded characteristic, that is, it is a ratio of two bounded analytic functions.
\end{theorem}
\begin{proof}
Let $B_k$ be the Blaschke product over the orbit of $c_{z_*,k}$, $k\ne 0$
\begin{equation*}%\label{28aug15}
B_k(z)=\prod\limits_{\gamma\in\Gamma}\frac{\gamma (z)-c_{z_*,k}}{\gamma (z)-\overline{c_{z_*,k}}}\ d_\gamma,\quad z\in\mathbb C_+,
\end{equation*}
where $|d_\gamma|=1$ are chosen so that the factors in $B_k$ are positive at $z_*$.
It converges since $\Gamma$ is of the convergent type.
We now consider
\begin{equation*}%\label{28aug16}
B(z)= \prod\limits_{k\ne 0}B_k(z).
\end{equation*}
This product converges due to assumption \eqref{28aug01}. Moreover, it converges uniformly on the compact subsets of $\mathbb C_+$.

The goal here is to prove that $$\frac{1}{(z-\overline{z_*})^2}\frac{B(z)}{g'(z, z_*)}$$ is a bounded analytic function on $\mathbb C_+$. Then $g'(z, z_*)$ will be the ratio of the following two bounded analytic functions
$$
\frac{1}{(z-\overline{z_*})^2}B(z)\quad
\text{and}\quad
\frac{1}{(z-\overline{z_*})^2}\frac{B(z)}{g'(z, z_*)}.
$$
More precisely, we will prove that
\begin{equation}\label{29aug11}
\left|\frac{1}{(z-\overline{z_*})^2}\frac{B(z)}{g'(z, z_*)}\right|\le 1,\quad z\in\mathbb C_+.
\end{equation}
It turns out that it is easier to prove even a stronger inequality
\begin{equation}\label{30aug01}
\ff(z)\le 1, \quad z\in\mathbb C_+, \quad \text{where}\quad \ff(z)=\left|\frac{B(z)}{g'(z, z_*)}\right|
\sum\limits_{\gamma\in\Gamma}
\dfrac{|\gamma'(z)|}{|\gamma (z)-\overline{z_*}|^2}.
\end{equation}
It is easier because of the automorphic property of the latter function. Note that the series in \eqref{30aug01}
converges to a function continuous on $\mathbb C_+$ for any group $\Gamma$ of convergent type.
So, we are going to prove that
\begin{equation}\label{30aug02}
\ff(z)=\sum\limits_{\gamma\in\Gamma}
\left|
\dfrac{B(z)\gamma'(z)}{g'(z, z_*)(\gamma(z)-\overline{z_*})^2}
\right|
\le 1,
\quad z\in\mathbb C_+.
\end{equation}
Observe that
$$
\dfrac{B(z)\gamma'(z)}{g'(z, z_*)(\gamma(z)-\overline{z_*})^2}
%=
%\dfrac{B(z)\gamma'(z)}{
%{(\gamma(z)-\overline{z_*})^2
%g_n(z, z_*)\sum\limits_{\gamma\in\Gamma}\frac{\gamma'(z)}{(\gamma(z)-z_*)(\gamma(z)-\overline{z_*})}}
%}
$$
is holomorphic on $\mathbb C_+$. Therefore, its absolute value is a
subharmonic function on $\mathbb C_+$. Hence $\ff(z)$
is a subharmonic function, which is automorphic with respect to
$\Gamma$.

We consider first the finitely generated approximation described in Section \ref{FZ}. Let $\O_n$ be the Denjoy domain
corresponding to the subgroup $\G_n$, $\L_n:\bbC_+/\G_n\simeq\O_n$.
Let $c_{z_*,k}^{(n)}$ be the
zero of $g'_n(z, z_*)$ on the $k$-th semicircle.
Let $B_k^{(n)}$ be the Blaschke product over the orbit of $c_{z_*,k}^{(n)}$ under $\Gamma_n$
$$
B_k^{(n)}(z)=\prod\limits_{\gamma\in\Gamma_n}\frac{\gamma(z)-c_{z_*,k}^{(n)}}{\gamma(z)-\overline{c_{z_*,k}^{(n)}}}\ d_\gamma,\quad z\in\mathbb C_+,
$$
if $k$-th semicircle is a part of the boundary of $\mathcal F_n$, and let $B_k^{(n)}(z)=1$ otherwise.
We now consider
$$
B^{(n)}(z)= \prod\limits_{k\ne 0}B_k^{(n)}(z).
$$
We are going to prove this approximative version of \eqref{30aug02}
\begin{equation}\label{30aug03}
\ff_n(z)\le 1, \quad z\in\mathbb C_+, \quad\text{where}\quad \ff_n(z)=\sum\limits_{\gamma\in\Gamma_n}
\left|
\dfrac{B^{(n)}(z)\gamma'(z)}{g'_n(z, z_*)(\gamma(z)-\overline{z_*})^2}
\right|.
\end{equation}
The advantage of the series in \eqref{30aug03} over the series in \eqref{30aug02} is that it
converges also on the boundary of the domain $\mathcal F_n$ and that the sum in \eqref{30aug03} is continuous
on $\mathcal F_n$ and up to the boundary,
since $\Gamma_n$ is finitely generated. The same is true for the fundamental domain of $\Gamma_n$, which is the union of $\mathcal F_n$ and the reflection of $\mathcal F_n$ about the $0$-th semicircle.

Due to the automorphic property of $\ff_n(z)$, it possesses the representation
\begin{equation*}%\label{}
\ff_n(z)=\fF_n(\L_n(z)),
\end{equation*}
where $\fF_n(\l)$
is still subharmonic in $\O_n$ and continuous
up to the boundary of the domain.  Therefore, its maximum is attained on the boundary of $\O_n$. Thus, going back to the function $\ff_n(z)$, we get that its maximum is attended  on the part of the boundary of the fundamental domain that lies on the real axis. Recall that on the boundary of the fundamental domain of $\Gamma_n$ all series below converge to continuous functions. Therefore, for real $z$ on the boundary of the fundamental domain of $\Gamma_n$ we have, by \eqref{28aug17},
\begin{equation}\label{30aug04}
|g'_n(z, z_*)|= \left|g_n(z, z_*)\sum\limits_{\gamma\in\Gamma_n}\frac{\gamma'(z)}{(\gamma(z)-z_*)(\gamma(z)-\overline{z_*})}\right|
=\sum\limits_{\gamma\in\Gamma_n}\frac{\gamma'(z)}{|\gamma(z)-\overline{z_*}|^2}
.
\end{equation}
Here we used the fact that $\gamma(z)$ is real for every real $z$ and that $\gamma'(z)$ is positive for every real $z$.
Therefore, for every real $z$ on the boundary of the fundamental domain
\begin{equation}\label{30aug05}
\ff_n(z)=\sum\limits_{\gamma\in\Gamma_n}
\left|
\dfrac{B^{(n)}(z)\gamma'(z)}{g'_n(z, z_*)(\gamma(z)-\overline{z_*})^2}
\right|
=\frac{1}{|g'_n(z, z_*)|}
\sum\limits_{\gamma\in\Gamma_n}\frac{\gamma'(z)}{|\gamma(z)-\overline{z_*}|^2}
=1.
\end{equation}
Hence, \eqref{30aug03} follows. Thus we have this approximative version of \eqref{30aug01}
\begin{equation}\label{30aug06}
\left|\frac{B^{(n)}(z)}{g'_n(z, z_*)}\right|
\sum\limits_{\gamma\in\Gamma_n}
\dfrac{|\gamma'(z)|}{|\gamma(z)-\overline{z_*}|^2}
\le 1,
\quad z\in\mathbb C_+.
\end{equation}

Now we want to pass to the limit in \eqref{30aug06} for arbitrary fixed $z\in\mathbb C_+$ as $n$ goes to infinity.
By Lemma \ref{L:28aug01}, $g'_n(z, z_*)$ converges to $g'(z, z_*)$.
The sum over $\Gamma_n$
converges to the sum over $\Gamma$.
It remains to show that $|B^{(n)}(z)|$ converges to $|B(z)|$.
Note that $|B^{(n)}_k(z)|=|g_n(c_{z_*,k}^{(n)}, z)|$ converges to
$|g(c_{z_*,k}, z)|=|B_k(z)|$, by Lemma \ref{L:28aug01}. % and Remark \ref{R:28aug01}.
Further,
$$
|B^{(n)}_k(z_*)|=|g_n(c_{z_*,k}^{(n)}, z_*)|\ge |g(c_{z_*,k}^{(n)}, z_*)| \ge |g(c_{z_*,k}, z_*)|.
$$
The first inequality holds since $g_n$ is a divisor of $g$, the second does since $c_{z_*,k}$ is the point of minimum of
$|g|$ on the $k$-th semi-circle.
By assumption \eqref{28aug01} the product
$$
\prod\limits_{k\ne 0}|g(c_{z_*,k}, z_*)|
$$
converges (that is greater than $0$). Then, by the Dominated Convergence theorem\footnote{This case reduces to the standard Dominated Convergence by applying $(-\log)$ to the products.},
\begin{align*}
\lim_{n\to\infty}|B^{(n)}(z_*)|=&\lim_{n\to\infty}\prod\limits_{k\ne 0}|B^{(n)}_k(z_*)|
=\prod\limits_{k\ne 0}\lim_{n\to\infty}|B^{(n)}_k(z_*)|
\\
=&
\prod\limits_{k\ne 0}|B_k(z_*)|=|B(z_*)|.
\end{align*}
There exists a subsequence $n_j$ such that $B^{(n_j)}(z)$
converges for all $z\in\mathbb C_+$. Let
$$
\widetilde B(z)=\lim_{j\to\infty}B^{(n_j)}(z).
$$
Pick and hold any $z\in\mathbb C_+$.
Then, by the Fatou's lemma\footnote{Same explanation as in the previous footnote.},
\begin{align*}
|\widetilde B(z)|&=\lim_{j\to\infty}|B^{(n_j)}(z)|=\lim_{j\to\infty}\prod\limits_{k\ne 0}|B^{(n_j)}_k(z)|
\\
&\le\prod\limits_{k\ne 0}\lim_{j\to\infty}|B^{(n_j)}_k(z)|
=\prod\limits_{k\ne 0}|B_k(z)|=|B(z)|.
\end{align*}
Thus
$$
|\widetilde B(z)|\le |B(z)|,\quad z\in\mathbb C_+.
$$
Since
$$
|\widetilde B(z_*)|= |B(z_*)|,
$$
the equality must hold
\begin{equation}\label{24oct01}
|\widetilde B(z)|= |B(z)|,\quad z\in\mathbb C_+.
\end{equation}
Since \eqref{24oct01} holds for every subsequential limit $\widetilde B(z)$ of $B^{(n)}(z)$,
we get
$$
 B(z)=\lim_{n\to\infty}B^{(n)}(z).
$$
Thus, we get \eqref{30aug01} and, therefore, \eqref{29aug11}.
\end{proof}
\if{%%%%%%%%%%%%%%%%%%%%%%%%%%%%%%%%%%%%%%%%%%%%%%%%%%%
\begin{remark}
Actually we proved that
\begin{equation}\label{31aug01}
\left|
\frac{B(z)}{g'(z, z_*)}
\sum\limits_{\gamma\in\Gamma}
\dfrac{\gamma'(z)}{(\gamma(z)-\overline{z_*})^2}
\right|
\le
\left|
\frac{B(z)}{g'(z, z_*)}
\right|
\sum\limits_{\gamma\in\Gamma}
\dfrac{|\gamma'(z)|}{|\gamma(z)-\overline{z_*}|^2}
\le 1,
\quad z\in\mathbb C_+.
\end{equation}
\end{remark}
}\fi%%%%%%%%%%%%%%%%%%%%%%%%%%%%%%%%%%%%%%%%%%%%%%%%%%
\begin{remark}\label{01sep03}
Since the function
\begin{equation*}%\label{30aug07}
\frac{1}{(z-\overline{z_*})^2}\frac{B(z)}{g'(z, z_*)}
\end{equation*}
is bounded, it can be written as
$$
\frac{1}{(z-\overline{z_*})^2}\frac{B(z)}{g'(z, z_*)}=I(z)\cdot O_2(z),
$$
where $I$ is an inner function and $O_2$ is a bounded outer function.
Moreover, $I$ is a singular inner function, since the left hand side does not have zeros in $\mathbb C_+$.
Therefore,
$$
g'(z, z_*)=\frac{O_1(z)}{O_2(z)}\frac{B(z)}{I(z)},
$$
where $O_1(z)=\frac{1}{(z-\overline{z_*})^2}$ is also a bounded outer function. Thus,
\begin{equation}\label{02sep10}
g'(z, z_*)=O(z)\frac{B(z)}{I(z)},
\end{equation}
where
$O(z)=\frac{O_1(z)}{O_2(z)}$ is a ratio of two bounded outer functions.
\end{remark}
\begin{theorem}[Pommerenke] \label{02sep04} The function
$$
\frac{1}{(z-\overline{z_*})^2}\frac{B(z)}{g'(z, z_*)}
$$
is outer. That is, $I(z)=1$.
\end{theorem}
\begin{lemma}\label{02sep01}
Let $x\in\mathbb R$.
The nontangential limits $g(x, z_*)$ and $g'(x, z_*)$ exist with $|g(x, z_*)|=1$, $g'(x, z_*)$ finite
if and only if
\begin{equation}\label{02sep02-2019}
\sum\limits_{\gamma\in\Gamma}
\dfrac{|\gamma'(x)|}{|\gamma(x)-\overline{z_*}|^2}<\infty.
\end{equation}
In this case
\begin{equation}\label{02sep02}
\dfrac{1}{i}\dfrac{g'(x, z_*)}{g(x, z_*)}=
|g'(x, z_*)|=2\Im z_*\sum\limits_{\gamma\in\Gamma}
\dfrac{|\gamma'(x)|}{|\gamma(x)-\overline{z_*}|^2}.
\end{equation}
Hence, in our case ($g$ is a Blaschke product, $g'$ is of bounded characteristic) \eqref{02sep02} holds almost everywhere on $\mathbb R$.
\end{lemma}
\begin{proof}
This lemma is Corollary \ref{19oct09} of the Appendix with $w=g(z,z_*)$, which is a product of these Blaschke factors
$$
B_\g(z)=\frac{\gamma(z)-z_*}{\gamma(z)-\overline{z_*}}.
$$
\eqref{02sep02} follows since in this case
$$
B'_\g(z)=2i\Im z_*\frac{\gamma'(z)}{(\gamma(z)-\overline{z_*})^2}.
$$
\end{proof}
\begin{lemma}\label{02sep06} For every $z\in\mathbb C_+$ the following inequality holds
\begin{equation}\label{21oct02}
\frac 1 \pi \int\limits_{\mathbb R}\log\sum\limits_{\gamma\in\Gamma}
\dfrac{|\gamma'(x)|}{|\gamma(x)-\overline{z_*}|^2}\frac{\Im z\, dx}{|x-z|^2}
\ge\log\sum\limits_{\gamma\in\Gamma}
\dfrac{|\gamma'(z)|}{|\gamma(z)-\overline{z_*}|^2}.
\end{equation}
\end{lemma}
\begin{proof}
Since
$$
\gamma'(z)=\frac{1}{(\gamma^{21}z+\gamma^{22})^2},
$$
one can write
$$
\sum\limits_{\gamma\in\Gamma}
\dfrac{|\gamma'(z)|}{|\gamma(z)-\overline{z_*}|^2}
=
\sum\limits_{\gamma\in\Gamma}\overline{\phi_{\gamma}(z)}\phi_\gamma(z),
$$
where
$$
\phi_{\gamma}(z)=\frac{1}{\gamma^{21}z+\gamma^{22}}\frac{1}{\gamma(z)-\overline{z_*}}.
$$
We enumerate the elements of the group $\G$, $\G=\{\g_k\}$, and consider functions $u_n(z)$ defined by the finite sums
\begin{equation*}%\label{21oct03}
u_n(z)
=\sum\limits_{k=1}^n
\frac{|\gamma'_k(z)|}{|\gamma_k(z)-\overline{z_*}|^2}
=
\sum\limits_{k=1}^n\overline{\phi_{\gamma_k}(z)}\phi_{\gamma_k}(z)
,
\quad  \Im z>0.
\end{equation*}
From here we see that $u_n$ is a subharmonic function since
\begin{equation*}%\label{21oct04}
\frac{\partial^2}{\partial z \partial \overline z}u_n(z)
= \sum\limits_{k=1}^n\overline{\phi'_{\gamma_k}(z)}\phi'_{\gamma_k}(z))\ge 0.
\end{equation*}
Also $\log u_n(z)$ is subharmonic, since
\begin{equation*}%\label{21oct05}
\frac{\partial^2}{\partial z \partial \overline z}\log u_n(z)
= -\frac{1}{u_n^2(z)}\frac{\partial u_n}{\partial z}
\frac{\partial u_n}{\partial \overline z}+\frac{1}{u_n}
\frac{\partial^2 u_n}{\partial z \partial \overline z}=
\end{equation*}
$$
\frac{1}{u_n^2(z)}
\left\{
\sum\limits_{k=1}^n\overline{\phi_{\gamma_k}(z)}\phi_{\gamma_k}(z)
\sum\limits_{k=1}^n\overline{\phi'_{\gamma_k}(z)}\phi'_{\gamma_k}(z)
-
\sum\limits_{k=1}^n\overline{\phi_{\gamma_k}(z)}\phi'_{\gamma_k}(z)
\sum\limits_{k=1}^n\overline{\phi'_{\gamma_k}(z)}\phi_{\gamma_k}(z)
\right\},
$$
which is nonnegative by the Cauchy-Schwarz inequality. Therefore,
$$
\frac 1 \pi \int\limits_{\mathbb R}\log
\sum\limits_{k=1}^n
\frac{|\gamma'_k(x)|}{|\gamma_k(x)-\overline{z_*}|^2}
\frac{\Im z\, dx}{|x-z|^2}
\ge
\log
\sum\limits_{k=1}^n
\frac{|\gamma'_k(z)|}{|\gamma_k(z)-\overline{z_*}|^2}
.
$$
We now pass to the limit in this inequality. Since all integrands here have lower summable bound
$\log\frac{1}{|x-\overline{z_*}|^2}$, %%%%% \ge \log\frac{1}{(\Im {z_*})^2}$,
the Monotone Convergence Theorem applies and we get \eqref{21oct02}.
\end{proof}
\begin{proof}[Proof of Theorem \ref{02sep04}]
It follows from
\begin{equation*}%\label{28aug02}
g'(z, z_*)= (z_*-\overline{z_*})\sum\limits_{\gamma\in\Gamma}
\frac{g(z, z_*)\gamma'(z)}{(\gamma(z)-z_*)(\gamma(z)-\overline{z_*})}
\end{equation*}
that for $z\in\mathbb C_+$
\begin{align}\nonumber\label{02sep05}
|g'(z, z_*)|= & 2\Im z_*
\left|
\sum\limits_{\gamma\in\Gamma}\frac{g(z, z_*)\gamma'(z)}{(\gamma(z)-z_*)(\gamma(z)-\overline{z_*})}
\right|
\\
\le &2\Im z_*
\sum\limits_{\gamma\in\Gamma}\left|\frac{g(z, z_*)\gamma'(z)}{(\gamma(z)-z_*)(\gamma(z)-\overline{z_*})}\right|
\le 2\Im z_*
\sum\limits_{\gamma\in\Gamma}
\dfrac{|\gamma'(z)|}{|\gamma(z)-\overline{z_*}|^2}.
\end{align}
We used here only one of the factors of $g(z, z_*)$ in every term. Now, by Lemmas \ref{02sep01} and \ref{02sep06},
\begin{align*}%\nonumber
\frac 1 \pi \int\limits_{\mathbb R}\log \frac{|g'(x, z_*)|}{2\Im z_*}\frac{\Im z}{|x-z|^2}dx
=&
\frac 1 \pi \int\limits_{\mathbb R}\log\sum\limits_{\gamma\in\Gamma}
\dfrac{|\gamma'(x)|}{|\gamma(x)-\overline{z_*}|^2}\frac{\Im z}{|x-z|^2}dx
\\
%\label{02sep07}
\ge & \log\sum\limits_{\gamma\in\Gamma}
\dfrac{|\gamma'(z)|}{|\gamma(z)-\overline{z_*}|^2}.
\end{align*}
That is,
$$
\frac 1 \pi \int\limits_{\mathbb R}\log |g'(x, z_*)|\frac{\Im z}{|x-z|^2}dx
\ge  \log 2\Im z_*\sum\limits_{\gamma\in\Gamma}
\dfrac{|\gamma'(z)|}{|\gamma(z)-\overline{z_*}|^2}.
$$
On the other hand
\begin{equation*}%\label{02sep08}
\frac 1 \pi\int\limits_{\mathbb R}\log |g'(x, z_*)|\frac{\Im z}{|x-z|^2}dx
=
\frac 1  \pi\int\limits_{\mathbb R}\log|O(x)|\frac{\Im z}{|x-z|^2}dx = \log|O(z)|,
\end{equation*}
since $O$ is a ratio of two bounded outer functions. Thus,
\begin{equation}\label{02sep09}
 2\Im z_*\sum\limits_{\gamma\in\Gamma}
\dfrac{|\gamma'(z)|}{|\gamma(z)-\overline{z_*}|^2}\le |O(z)|,\quad z\in \mathbb C_+.
\end{equation}
Combining \eqref{02sep09} and \eqref{02sep05}, we get
$$
|g'(z,z_*)|\le |O(z)|,\quad z\in \mathbb C_+.
$$
That is, in view of \eqref{02sep10},
$$
\left|\frac{B(z)}{I(z)}\right|
=
\left|\frac{g'(z,z_*)}{O(z)}\right|
\le 1.
$$
The latter implies that $I(z)=1$.
\end{proof}

\section{Conditions ($b$) and (B)  in Theorem \ref{mth-15oct01}.}\label{sec5}

\subsection{ Martin Function with a pure point measure}\label{DM}

Recall that  (see \eqref{11sep7}, \eqref{27dec06}) $M(\l)=\Im\t(\l)$, $\l\in\O$ and that $m(z)=\t(\Lambda(z))$. Thus,
$$
M(\L(z))=\Im m(z).
$$
$m(z)$ is a  single-valued holomorphic function  defined in $\mathbb C_+$, additively character automorphic with respect to $\Gamma$.
%%  Since , $m(z)$ has positive imaginary part when $z$ is in the fundamental domain. Since $m(z)$ is ,
$\Im m(z)\ge 0$ for all $z\in\bbC_+$.
Therefore, $m(z)$ admits  a Riesz-Herglotz representation
\begin{equation*}%\label{24aug01}
m(z)=az + b +
\int\limits_{\mathbb R}
\left(
\frac{1}{x-z}
-\frac{x}{1+x^2}
\right)
\sigma(dx),
\end{equation*}
where $a\ge 0$, $b$ is real and $\sigma$ is a singular measure on $\mathbb R$ with
$$
\int\limits_{\mathbb R}
\frac{\sigma(dx)}{1+x^2}<\infty.
$$
Let us mention that $e^{i\ell m(z)}$ is a singular inner character automorphic function, for all $\ell>0$.

We observe (see, e.g.  \cite{VYM}) that for Martin functions in Denjoy domains there are two options: either $a>0$ (that is, ($b_1$) of \eqref{27dec04} holds)
and $\sigma$ is a pure point measure (that is, (b) of Theorem \ref{mth-15oct01} holds), supported by orbits of $\infty$ and $0$;
or $a=0$ (that is, ($b_1$) of \eqref{27dec04} fails) and $\sigma$ is a continuous singular measure (that is, (b) of Theorem \ref{mth-15oct01} fails).
For further references we state it as
\begin{proposition}\label{27dec05}
Properties (b) of Theorem \ref{mth-15oct01} and ($b_1$) of \eqref{27dec04} are equivalent.
\end{proposition}
\noindent
We point out that the orbits $\{\g(0)\}_{\g\in\G}$ and $\{\g(\infty)\}_{\g\in\G}$ cannot intersect due to the structure of the generators of the group $\Gamma$.

We start with a singular function supported by the orbit of $\infty$,
\begin{equation}\label{25aug01}
m_+(z)=z+
\sum\limits_{\gamma\in\Gamma\ \gamma\ne 1}
\left(
\frac{1}{\gamma(\infty)-z}
-\frac{\gamma(\infty)}{1+\gamma(\infty)^2}
\right)
\sigma_\gamma
\end{equation}
where
\begin{equation}\label{22oct01}
\sum\limits_{\gamma\in\Gamma\ \gamma\ne 1}
\frac{\sigma_\gamma}{1+\gamma(\infty)^2}
<\infty.
\end{equation}
\begin{lemma}\label{l230ct}
The function $m_+(z)$ defined in \eqref{25aug01} is additive character automorphic with respect to the group $\G$ if and only if
\begin{equation}\label{25aug01-tilde3}
\sigma_\gamma=\frac{1}{(\gamma^{21})^2}.
\end{equation}
Respectively,
\begin{equation}\label{12oct06}
\sum\limits_{\gamma\in\Gamma}
\frac{\sigma_\gamma}{1+\gamma(\infty)^2}
=
\sum\limits_{\gamma\in\Gamma}
\frac{1}{(\gamma^{11})^2+(\gamma^{21})^2}
<\infty,
\end{equation}
and
\begin{equation}\label{10oct02}
m_+(z)=\sum\limits_{\gamma\in\Gamma}
\left(\gamma(z)-\Re\gamma(i)\right)
.
\end{equation}
\end{lemma}

\begin{proof}
Since
$$
\gamma(z)=\frac{\gamma^{11}z+\gamma^{12}}{\gamma^{21}z+\gamma^{22}},\quad \gamma\in SL_2(\mathbb R),
$$
we have
$
\gamma(\infty)={\gamma^{11}}/{\gamma^{21}}.
$
Note that for every $\gamma\in\Gamma$, $\gamma^{11}\ne 0$ (since $\infty$ is not carried to $0$) and $\gamma^{22}\ne 0$ (since $0$ is not carried to $\infty$); also for $\gamma\ne 1$, $\gamma^{12}\ne 0$
(since $0$ is not a fixed point) and $\gamma^{21}\ne 0$ (since $\infty$ is not a fixed point).
So, let $m_+(z)$ be defined by \eqref{25aug01}
\begin{equation}\label{25aug01-tilde}
m_+(z)=z+
\sum\limits_{\widetilde\gamma\in\Gamma\ \widetilde\gamma\ne 1}
\left(
\frac{1}{\widetilde\gamma(\infty)-z}
-\frac{\widetilde\gamma(\infty)}{1+\widetilde\gamma(\infty)^2}
\right)
\sigma_{\widetilde\gamma}.
\end{equation}
Let $\gamma\in\Gamma$, then $m_+(\gamma(z))$ is the same as $m_+(z)$ up to a real additive constant.
Let $\gamma\ne 1_\G$. We substitute $\gamma(z)$ instead of $z$ in \eqref{25aug01-tilde} and
we consider the term with $\widetilde\gamma=\gamma$. % in $m_+(\gamma(z))$.
We have
$$
\frac{1}{\gamma(\infty)-\gamma(z)}=
\left({\frac{\gamma^{11}}{\gamma^{21}}-\frac{\gamma^{11}z+\gamma^{12}}{\gamma^{21}z+\gamma^{22}}}\right)^{-1}=
(\gamma^{21})^2 z + \gamma^{21}\gamma^{22}.
$$
Since the coefficient of $z$ in $m(\gamma(z))$ must be equal to $1$,
we get \eqref{25aug01-tilde3}; then \eqref{12oct06} follows from \eqref{22oct01}.
Thus, we can write
\be\label{18dec01}
m_+(z)=
z+\sum\limits_{\gamma\in\Gamma\ \gamma\ne 1}\left(
\frac{1}{\gamma(\infty)-z}
-\frac{\gamma(\infty)}{1+\gamma(\infty)^2}
\right)
\frac{1}{(\gamma^{21})^2}
\ee
$$
=
z+\sum\limits_{\gamma\in\Gamma\ \gamma\ne 1}\left(
\frac{1}{\gamma^{11}-\gamma^{21}z}\
\frac{1}{\gamma^{21}}
-\frac{\gamma(\infty)}{1+\gamma(\infty)^2}\
\frac{1}{(\gamma^{21})^2}
\right).
$$
Since $\g\in SL_2(\bbR)$, we have
\begin{equation*}%\label{10oct01}
\gamma^{-1}=\begin{bmatrix}\gamma^{22} & -\gamma^{12}\\ -\gamma^{21} & \gamma^{11}\end{bmatrix}\quad
\text{for}\quad
\gamma=\begin{bmatrix}\gamma^{11} & \gamma^{12}\\ \gamma^{21} & \gamma^{22}\end{bmatrix}.
\end{equation*}
Then we can further rewrite
$$
m_+(z)=
z+\sum\limits_{\gamma\in\Gamma\ \gamma\ne 1}\left(
\gamma^{-1}(z)-\gamma^{-1}(\infty)
-\frac{\gamma(\infty)}{1+\gamma(\infty)^2}
\frac{1}{(\gamma^{21})^2}
\right)
$$
\be\label{18dec02}
=
\sum\limits_{\gamma\in\Gamma}
\left(\gamma^{-1}(z)-c(\gamma^{-1})\right)
=
\sum\limits_{\gamma\in\Gamma}
\left(\gamma(z)-c(\gamma)\right)
,
\ee
where $c(1_\G)=0$ and for $\gamma\ne 1_\G$
$$
c(\gamma)=
\gamma(\infty)
+\frac{\gamma^{-1}(\infty)}{1+\gamma^{-1}(\infty)^2}
\frac{1}{(-\gamma^{21})^2}
=\frac{\gamma^{11}}{\gamma^{21}}-\frac{\gamma^{22}}{\gamma^{21}}
\cdot\frac{1}{(\gamma^{22})^2+(\gamma^{21})^2}.
$$
Actually, since
$$
\Re\left(
\frac{1}{\gamma(\infty)-i}
-\frac{\gamma(\infty)}{1+\gamma(\infty)^2}
\right)
=
\Re\left(
\frac{\gamma(\infty)+i}{1+\gamma(\infty)^2}
-\frac{\gamma(\infty)}{1+\gamma(\infty)^2}
\right)
=0,
$$
we get from \eqref{18dec01} and \eqref{18dec02} that
$$
c(\g)=\Re\gamma(i)=\frac{\gamma^{12}\gamma^{22}+\gamma^{11}\gamma^{21}}{(\gamma^{22})^2+(\gamma^{21})^2}.
$$
That is, we get \eqref{10oct02}.
The fact that the function $m_+(z)$ is additively character automorphic follows directly from representation
\eqref{10oct02}.
\end{proof}

The convergence in \eqref{10oct02} is absolute and uniform as long as $z$ is bounded away from the orbit of $\infty$.
We also see that
$$
m'_+(z)=\sum\limits_{\gamma\in\Gamma}\gamma'(z)
.
$$
the convergence here is also absolute and uniform as long as $z$ is bounded away from the orbit of $\infty$.
By  \eqref{10oct02}, we get
\begin{equation*}%\label{12oct01}
\Im m_+(z)=
\sum\limits_{\gamma\in\Gamma}\Im\gamma(z)
=
\sum\limits_{\gamma\in\Gamma}
\frac{\Im z}{|\g^{21}z+\g^{22}|^2}
=\Im z\sum\limits_{\gamma\in\Gamma}|\g'(z)|.
\end{equation*}
Thus
\begin{equation}\label{2sepp}
\frac{\Im m_+(z)}{\Im z}
=\sum\limits_{\gamma\in\Gamma}
\frac{1}{|\g^{21}z+\g^{22}|^2}
=\sum\limits_{\gamma\in\Gamma}|\g'(z)|
.
\end{equation}

The antiholomorphic automorphism $\l\mapsto \overline{\l}$ on $\O$ acts as
$z \mapsto 1/\overline{z}$  on the universal covering $\mathbb C_+$.
Thus, the symmetric Martin function $m(z)$ of the group $\G$ possesses
the following property
$$
\Im m(1/\overline z)=\Im m(z).
$$
Let us   define $m_-(z)$ so that
$$
\Im m_-(z)= \Im m_+(1/\overline z)
$$
In this case,
\begin{equation}\label{12oct02}
\Im m_-(z)= \Im m_+(1/\overline z)
=
\sum\limits_{\gamma\in\Gamma}
\frac{\Im 1/\overline z}{\left|\dfrac{\g^{21}}{\overline z}+\g^{22}\right|^2}
=
\sum\limits_{\gamma\in\Gamma}
\frac{\Im z}{|\g^{21}+\g^{22}z|^2}.
\end{equation}
\begin{lemma}\label{18dec03}
Group $\Gamma$ has the following symmetry. Let
\begin{equation}\label{26sep13}
\tau(z)=\frac{1}{\overline z}.
\end{equation}
This is a reflection about the unit circle. Also $\tau:\mathbb C_+\to\mathbb C_+$. Then
\begin{equation}\label{26sep14}
\gamma\mapsto \widetilde\gamma=\tau\gamma\tau
\end{equation}
is an automorphism (one-to-one and onto) of $\Gamma$.
In the matrix form this automorphism reads as follows
\begin{equation}\label{26sep08}
\gamma=\begin{bmatrix}\gamma^{11} & \gamma^{12}\\ \gamma^{21} & \gamma^{22}\end{bmatrix}
\mapsto
\widetilde\gamma=\begin{bmatrix}\gamma^{22} & \gamma^{21}\\ \gamma^{12} &  \gamma^{11}\end{bmatrix}.
\end{equation}
\end{lemma}
\begin{proof}
Lemma follows from the observation that every generator of the group $\Gamma$ is a composition of two reflections about the boundary semicircles of $\mathcal F$.
\end{proof}
In view of this lemma, we may continue \eqref{12oct02} (compare with \eqref{2sepp}) as
\begin{equation*}%\label{12oct03}
\Im m_-(z)
=
\sum\limits_{\gamma\in\Gamma}
\frac{\Im z}{|\g^{21}+\g^{22}z|^2}
=
-\sum\limits_{\gamma\in\Gamma}
\Im\frac{1}{\widetilde\gamma(z)}
=
-\sum\limits_{\gamma\in\Gamma}
\Im\frac{1}{\gamma(z)}
.
\end{equation*}
Therefore,
$$
m_-(z)=-\sum\limits_{\gamma\in\Gamma}
\left(\frac{1}{\gamma(z)}
-
\Re\frac{1}{\gamma(i)}\right).
$$
This function also admits representation of type \eqref{25aug01}
\begin{equation}\label{12oct04}
m_-(z)=
\sum\limits_{\gamma\in\Gamma}
\left(
\frac{1}{\gamma(0)-z}
-\frac{\gamma(0)}{1+\gamma(0)^2}
\right)
\beta_\gamma,\quad \text{where}\ \beta_\gamma=\frac{1}{(\gamma^{22})^2}.
\end{equation}
The corresponding convergence condition reads as follows
\begin{equation}\label{12oct05}
\sum\limits_{\gamma\in\Gamma}
\frac{\beta_\gamma}{1+\gamma(0)^2}
=
\sum\limits_{\gamma\in\Gamma}
\frac{1}{(\gamma^{12})^2+(\gamma^{22})^2}<\infty
,
\end{equation}
which is equivalent to \eqref{12oct06} in view of Lemma \ref{18dec03}.
Finally, we arrive at the following proposition.
\begin{proposition}\label{pr23oct}
($b$) in Theorem \ref{mth-15oct01} holds if and only if
\begin{equation}\label{23oct1}
\sum_{\g\in\G}{\Im \g(i)}=\sum_{\g\in\G}{|\g'(i)|}<\infty.
\end{equation}
In this case the symmetric Martin function of the group $\G$  is given by
\begin{equation}\label{26sep04}
m(z)=m_+(z)+m_-(z)=\sum\limits_{\gamma\in\Gamma}
\left(\gamma(z)-\frac{1}{\gamma(z)}\right)
-
\Re\left(
\gamma(i)
-\frac{1}{\gamma(i)}\right).
\end{equation}
Moreover,
\begin{equation}\label{26sep05}
m'(z)=\sum\limits_{\gamma\in\Gamma}\gamma'(z)+\sum\limits_{\gamma\in\Gamma}
\frac{\gamma'(z)}{\gamma^2 (z)}
%\label{26sep16}
=
\sum\limits_{\gamma\in\Gamma}
\frac{1}{(\g^{21}z+\g^{22})^2}
+
\sum\limits_{\gamma\in\Gamma}
\frac{1}{(\g^{11}z+\g^{12})^2}
\end{equation}
%\begin{align}\label{26sep05}
%m'(z)=&\sum\limits_{\gamma\in\Gamma}\gamma'(z)+\sum\limits_{\gamma\in\Gamma}
%\frac{\gamma'(z)}{\gamma^2 (z)}
%\\
%\label{26sep16}
%=&
%\sum\limits_{\gamma\in\Gamma}
%\frac{1}{(\g^{21}z+\g^{22})^2}
%+
%\sum\limits_{\gamma\in\Gamma}
%\frac{1}{(\g^{11}z+\g^{12})^2}
%\end{align}
and
\begin{equation}\label{26sep06}
\frac{\Im m(z)}{\Im z}
%\label{26sep07}
=
\sum\limits_{\gamma\in\Gamma}|\g'(z)|
+
\sum\limits_{\gamma\in\Gamma}\left|\frac{\g'(z)}{\g^2(z)}\right|
=
\sum\limits_{\gamma\in\Gamma}
\frac{1}{|\g^{21}z+\g^{22}|^2}
+
\sum\limits_{\gamma\in\Gamma}
\frac{1}{|\g^{11}z+\g^{12}|^2}.
\end{equation}
\end{proposition}
\begin{proof}\label{26sep15}
We consider
\begin{equation*}%\label{26sep11}
\gamma=\begin{bmatrix}\gamma^{11} & \gamma^{12}\\ \gamma^{21} & \gamma^{22}\end{bmatrix}
,\quad
\widetilde\gamma=\begin{bmatrix} \gamma^{22} & \gamma^{21}\\ \gamma^{12} &  \gamma^{11}\end{bmatrix},
\end{equation*}
\begin{equation*}%\label{26sep12}
\gamma^{-1}=\begin{bmatrix}\gamma^{22} & -\gamma^{12}\\ -\gamma^{21} & \gamma^{11}\end{bmatrix}
,\quad
\widetilde\gamma^{-1}=\begin{bmatrix}\ \ \gamma^{11} & -\gamma^{21}\\ -\gamma^{12} & \ \ \gamma^{22}\end{bmatrix}.
\end{equation*}
The meaning of $\widetilde\gamma$ is explained in \eqref{26sep13} - \eqref{26sep08}.
By looking at the first columns of those matrices, we can conclude that
one of the four convergence conditions below imply the others
\begin{equation}\label{26sep09}
\sum\limits_{\gamma\in\Gamma}
\frac{1}{(\gamma^{11})^2+(\gamma^{21})^2}<\infty
,\quad
\sum\limits_{\gamma\in\Gamma}
\frac{1}{(\gamma^{12})^2+(\gamma^{22})^2}
<\infty,
\end{equation}
\begin{equation}\label{26sep10}
\sum\limits_{\gamma\in\Gamma}
\frac{1}{(\gamma^{22})^2+(\gamma^{21})^2}<\infty
,\quad
\sum\limits_{\gamma\in\Gamma}
\frac{1}{(\gamma^{11})^2+(\gamma^{12})^2}
<\infty.
\end{equation}
In particular, \eqref{12oct05} and \eqref{12oct06} are equivalent. \eqref{23oct1} corresponds to the first condition in \eqref{26sep10}. Due to Lemma \ref{l230ct} we get representation \eqref{10oct02} for $m_+(z)$.
By \eqref{12oct05} we have \eqref{12oct04} for $m_-(z)$, and therefore
 \eqref{26sep04} - \eqref{26sep06}. Also in this formulas
$\gamma$ can be replaced with $\gamma^{-1}$, $\widetilde\gamma$ or $\widetilde\gamma^{-1}$ if needed.
\end{proof}

\if{%%%%%%%%%%%%%%%%%%%%%%%%%%%%%%
\begin{corollary}\label{cal}
Condition ($b_1$) of \eqref{27dec04} implies  \eqref{alco}.
\end{corollary}

\begin{proof}
In $\O=\mathbb C\setminus E$ we have two different  Martin functions $M_\pm$ such that
$$
M_\pm(\L(z))=\Im m_\pm(z).
$$
It is known that if the cone of Martin functions in the domain is two dimensional, then \eqref{alco} holds. Indeed,
consider functions $M_\pm(\l)$ on the upper half plane in $\mathbb C\setminus E$. As positive harmonic functions continuous up to the real line, they admit Poisson representations in terms of their values on the real line.
On $E$ they both vanish, on $\mathbb R\setminus E$ they coincide, since $M_-(\l)=M_+(\overline\l)$. Therefore, coefficients $a_\pm$ of $\Im\lambda$ in their Poisson representations are not equal. Hence at least one of them is positive. Finally we conclude that coefficient $a_++a_-$ of
$\Im\lambda$ in the Poisson representation of the symmetric Martin function $M(\lambda)=M_+(\lambda)+M_-(\lambda)$ is positive, that is \eqref{alco} holds.
\end{proof}
}\fi%%%%%%%%%%%%%%%%%%%%%%%%%%%%%%%%%

\begin{lemma}\label{L:29aug01}
Assume that convergence condition \eqref{12oct05} (or any equivalent) holds. Let $m$ be the symmetric Martin function of the group $\Gamma$ defined as in \eqref{26sep04}. Let $m_n(z)$ be the symmetric Martin function of the group $\Gamma_n$. All critical points of $m(z)$ and $m_n(z)$ are located on the boundary semi-circles of $\mathcal F$ and $\mathcal F_n$, respectively. Let $c_{k}^{(n)}$ be the zero of $m'_n(z)$ on the $k$-th semicircle and let $c_{k}$ be the zero of $m'(z)$ on the $k$-th semicircle. Then, as $n$ goes to $\infty$, $m_n(z)$ converges to $m(z)$ uniformly on the compact subsets in $\mathbb C_+$ and
$m'_n(z)$ converges to $m'(z)$ uniformly on the compact subsets in $\mathbb C_+$.  Also
$$
c_{k}^{(n)}\to c_{k}
$$
for every $k$.
%% , including $k=0$.
Moreover,
$m_n(c_{k}^{(n)})$ converges to $m(c_{k})$ and
$g_n(c_{k}^{(n)}, z_*)$ converges to $g(c_{k}, z_*)$ for every
$k$.
\end{lemma}
\begin{proof}
Completely parallel to the proof of Lemma \ref{L:28aug01}.
\end{proof}

\subsection{Akhiezer - Levin
condition}\label{Akhiezer-Levin}

In this subsection we prove

\begin{theorem}\label{ths}
Properties $(B)$ of Theorem \ref{mth-15oct01} and $(b_1)$ of \eqref{27dec04} are equivalent.
\end{theorem}

We start with the following lemma.

\begin{lemma}\label{23nov01}
Let $u(z)$ be a function with positive imaginary part on the upper half plane. We assume that $u$ is not a real constant.
Let
$$
f(z)=\frac{u(z)-i}{u(z)+i},\quad \text{respectively},\quad u(z)=i\frac{1+f(z)}{1-f(z)}.
$$
Let
$$
\zeta=\frac{z-i}{z+i},\quad \text{respectively},\quad z=i\frac{1+\zeta}{1-\zeta}.
$$
Let also
\be\label{27nov5}
 w(\zeta)=f\left(i\frac{1+\zeta}{1-\zeta}\right).
\ee
Then
$$
a :=\lim_{z=iy, y\to\infty}\frac{\Im u(z)}{\Im z} > 0
$$
if and only if
\be\nonumber %%\label{20dec01}
\lim_{\zeta >0, \zeta\to 1}  w(\zeta)=1\quad\text{and}\quad
d :=\lim_{\zeta >0, \zeta\to 1}\frac{1-| w(\zeta)|^2}{1-|\zeta|^2}<\infty.
\ee
In this case
$$
a=\frac{1}{d}.
$$
\end{lemma}
\begin{proof}
Note that $\zeta\to 1$ as $z\to\infty$.
Compute
$$
\frac{\Im u(z)}{\Im z}=\frac{1-|f(z)|^2}{|1-f(z)|^2}\frac{|1-\zeta|^2}{1-|\zeta|^2}
=\frac{1-|w(\zeta)|^2}{1-|\zeta|^2}\left|\frac{1-\zeta}{1-w(\zeta)}\right|^2.
$$
The backwards computation gives
\be\label{27nov7}
\frac{1-| w(\zeta)|^2}{1-|\zeta|^2}=
\frac{\Im u(z)}{\Im z}\frac{|z+i|^2}{|u(z)+i|^2}.
\ee
The limit
$$
a=\lim_{z=iy, y\to\infty}\frac{\Im u(z)}{\Im z} \ge 0
$$
always exists and is finite ($a$ is the coefficient of $z$ in the Riesz-Herglotz representation of $u$). Assume that $a>0$. Then
$$
\lim_{z=iy, y\to\infty}u(z)=\infty\
\text{
and, therefore,}\
\lim_{\zeta >0, \zeta\to 1}  w(\zeta)=1.
$$
Further,
\begin{align*}
d=\lim_{\zeta >0, \zeta\to 1}\frac{1-| w(\zeta)|^2}{1-|\zeta|^2}=&
\lim_{z=iy, y\to\infty}\frac{\Im u(z)}{\Im z}
\frac{|z+i|^2}{|u(z)+i|^2}\\
=&
 \lim_{z=iy, y\to\infty}\frac{\Im u(z)}{\Im z}\frac{(\Im z+1)^2}{(\Im u(z)+1)^2+(\Re u(z))^2}
\\
\le &
 \lim_{z=iy, y\to\infty}\frac{\Im u(z)}{\Im z}\frac{(\Im z+1)^2 }{(\Im u(z)+1)^2}=\frac{1}{a}<\infty.
\end{align*}
Conversely, let
$$
\lim_{\zeta >0, \zeta\to 1} w(\zeta)=1\quad\text{and}\quad
\lim_{\zeta >0, \zeta\to 1}\frac{1-|w(\zeta)|^2}{1-|\zeta|^2}=d<\infty.
$$
$d>0$, since $w$ is not a unimodular constant (since $u$ is not a real constant). Then, by the Julia Theorem
(see Theorem \ref{03sep01} in Appendix),
$$
\frac{1-|w(\zeta)|^2}{1-|\zeta|^2}\left|\frac{1-\zeta}{1-w(\zeta)}\right|^2\ge \frac{1}{d}>0.
$$
Therefore,
$$
\frac{\Im u(z)}{\Im z}
=\frac{1-|w(\zeta)|^2}{1-|\zeta|^2}\left|\frac{1-\zeta}{1-w(\zeta)}\right|^2\ge\frac{1}{d}>0
$$
and
$$
a=\lim_{z=iy, y\to\infty}\frac{\Im u(z)}{\Im z}\ge\frac{1}{d}>0.
$$
\end{proof}
Combining Lemma \ref{23nov01} and the Carath\' eodory- Julia theorem (Theorem \ref{03sep01} in Appendix) we get
\begin{corollary}\label{27nov6}
Assume that $u(z)$ is not a real constant, then
the following are equivalent
\begin{eqnarray*}
&&(1)\quad \text{There\ exists\ a\ sequence\ } z_k,\ \Im z_k>0,\ \lim z_k=\infty
\text{\ such\ that\ }
\nonumber\\
&&
\quad\quad
{\displaystyle\lim u(z_k)}=\infty\quad\text{and}\quad
{\displaystyle d_1 :=\lim
\frac{\Im u(z_k)}{\Im z_k}\frac{|z_k+i|^2}{|u(z_k)+i|^2}}<\infty;\nonumber \\
&&(2) \quad
{\displaystyle\lim_{z=iy, y\to\infty}u(z)}=\infty\quad\text{and}\quad
{\displaystyle d_2 :=\lim_{z=iy, y\to\infty}
\frac{\Im u(z)}{\Im z}\frac{|z+i|^2}{|u(z)+i|^2}}<\infty
;\nonumber\\
&&(3) \quad a=\lim_{z=iy, y\to\infty}\frac{\Im u(z)}{\Im z} > 0;\nonumber\\
\end{eqnarray*}
When these conditions hold, we have $d_1=d_2=\frac{1}{a}$. Function $u$ is a real constant if and only if
$d_1=d_2=0$.
\end{corollary}
\begin{proof}
Note first that, in view of formulas \eqref{27nov5} and \eqref{27nov7}, (1) and (2) are equivalent, respectively, to
\begin{eqnarray*}
&&(1')\quad \text{There\ exists\ a\ sequence\ } \z_k,\ |\z_k|<1,\ \lim\z_k=1
\text{\ such\ that\ }
\nonumber\\
&&
\quad\quad\ \
{\displaystyle\lim w(\z_k)}=1\quad\text{and}\quad
{\displaystyle\lim
\frac{1-| w(\zeta_k)|^2}{1-|\zeta_k|^2}}<\infty;\nonumber \\
&&(2') \
\lim_{\zeta >0, \zeta\to 1} w(\zeta)=1\quad\text{and}\quad
\lim_{\zeta >0, \zeta\to 1}\frac{1-|w(\zeta)|^2}{1-|\zeta|^2}<\infty.
\nonumber\\
\end{eqnarray*}
 By Lemma \ref{23nov01}, (3) is equivalent to (2') and, therefore (3) and (2) are equivalent.
 (2) obviously implies (1). It remains to show that (1) implies (2), equivalently, that (1') implies (2').
  By Theorem \ref{03sep01}, the second part of (1') implies the second part of (2').
Also (1') implies  \eqref{28oct01} of Theorem \ref{03sep01} with $w_0=1$. Hence, $w_0=1$ is the
limit of $w(\z)$ as $\z$, $|\z|<1$, approaches $t_0=1$ nontangentially.
%%% $$
%%% \frac{1-| w(\zeta)|^2}{1-|\zeta|^2}=
%%% \frac{\Im u(z)}{\Im z}\frac{|z+i|^2}{|u(z)+i|^2}.
%%% $$
\end{proof}

\begin{proof}[\bf Proof of Theorem \ref{ths}]
First we show that (B) of Theorem \ref{mth-15oct01} implies ($b_1$) of \eqref{27dec04}.
Property (B) of Theorem \ref{mth-15oct01} says that
\be\label{27nov2}
\lim_{\l=i\eta, \eta\to\infty}\frac{\Im\theta(\l)}{\Im \l} > 0.
\ee
By Corollary \ref{27nov6}, we have that
\be\label{18dec04}
\lim_{\l=i\eta, \eta\to\infty}\theta(\l)=\infty.
\ee
and
\be\label{28nov01}
{\displaystyle\lim_{\l=i\eta, \eta\to\infty}\frac{\Im \theta(\l)}{\Im \l}\frac{|\l+i|^2}{|\theta(\l)+i|^2}}<\infty.
\ee
Consider $z=\L^{-1}(\l)$ as a mapping from $\bbC_+\subset\bbC\setminus E$ to the fundamental domain (in $\bbC_+$).
Since $\L^{-1}(\l)$ is a nonconstant function with positive imaginary part, we get, by Corollary \ref{27nov6},
that
\be\label{28nov02}
{\displaystyle\lim_{\l=i\eta, \eta\to\infty}\frac{\Im \L^{-1}(\l)}{\Im \l}\frac{|\l+i|^2}{|\L^{-1}(\l)+i|^2}}>0.
\ee
Combining \eqref{28nov01} and \eqref{28nov02}, we get
\be\label{28nov03}
{\displaystyle\lim_{\l=i\eta, \eta\to\infty}\frac{\Im\theta(\l)}{\Im \L^{-1}(\l)}
\frac{|\L^{-1}(\l)+i|^2}{|\theta(\l)+i|^2}}<\infty.
\ee
Substituting $\l=\L(z)$, we get that
\be\label{28nov04}
{\displaystyle\lim\frac{\Im m(z)}{\Im z}\frac{|z+i|^2}{|m(z)+i|^2}}<\infty
\ee
as $z$ goes to $\infty$ along $\L^{-1}(\{i\eta, \eta>0\})$. We also have from \eqref{18dec04} that
$$
\lim m(z)=\infty
$$
as $z$ goes to $\infty$ along $\L^{-1}(\{i\eta, \eta>0\})$.
By Corollary \ref{27nov6},
$$
\lim_{z=iy, y\to\infty}\frac{\Im m(z)}{\Im z} > 0,
$$
which is ($b_1$) of \eqref{27dec04}.

%%%%%%%
Now we will show that
condition ($b_1$) of \eqref{27dec04} implies (B) of Theorem \ref{mth-15oct01}.
In
$\O=\mathbb C\setminus E$ we have two different  Martin functions $M_\pm$ such that
$
M_+(\L(z))=\Im m_+(z),
$
where the measure of $m_+$ is supported by the orbit $\{\g(\infty)\}$,  and $M_-(\l)=M_+(\overline\l)$.
It is known that if the cone of Martin functions in the domain is two dimensional, then (B) of Theorem \ref{mth-15oct01} holds. Indeed,
consider functions $M_\pm(\l)$ on the upper half plane in $\mathbb C\setminus E$. As positive harmonic functions continuous up to the real line, they admit Poisson representations in terms of their values on the real line.
On $E$ they both vanish, on $\mathbb R\setminus E$ they coincide. Therefore, coefficients $a_\pm\ge 0$ of $\Im\lambda$ in their Poisson representations are not equal. Hence at least one of them is positive. Finally we conclude that the coefficient $a_++a_-$ of $\Im\lambda$ in the Poisson representation of the symmetric Martin function $M(\lambda)=M_+(\lambda)+M_-(\lambda)$ is positive, that is, (B) holds.
\end{proof}
\begin{remark}
Observe that actually $a_-=0$, for otherwise, by Corollary \ref{27nov6},
$$
\lim_{z=iy, y\to\infty}\frac{\Im m_-(z)}{\Im z} > 0,
$$
which is not the case. Therefore, $a_+>0$. Now we can write
$$
M_+(\l)-M_-(\l)=a_+\Im\l
$$
or, after substituting $\l=\L(z)$,
$$
\Im m_+(z)-\Im m_-(z)=a_+\Im\L(z).
$$
Hence,
$$
\dfrac{1}{\lim\limits_{z=iy, y\to\infty}\dfrac{\Im\L(z)}{\Im z}}=a_+>0.
$$
%% If condition (B) fails, then $m^{(n)}_+$ and $m^{(n)}_-$ diverge, but their difference still has a finite limit.
%% Therefore, after renormalization the limit of the difference is $0$.

\end{remark}

\section{Proof of the main theorem (Theorem \ref{mth-15oct01})}\label{pro}

In this section we will use restatement of condition (A) in terms of the universal cover
\begin{equation}\label{17oct03}
(A)\qquad
\prod_{k\ne 0}|g(c_{k}, z_*)|>0,
\end{equation}
where $c_{k}$ are the zeros of $m'(z)$ (one on each boundary semicircle of $\cF$). This is the Blaschke condition on {\it all} the zeros of $m'(z)$ in the upper half plane (orbits of $c_k$ under the action of the group $\G$).

\subsection{Proof of the implication $(ii)\Rightarrow (iii)$}

By Proposition \ref{27dec05},
($b$) is equivalent to ($b_1$) of \eqref{27dec04} and, by Theorem \ref{ths}, ($b_1$) is equivalent to (B) of Theorem \ref{mth-15oct01}. Property ($a$) implies $(A)$, since zeros of a function of bounded characteristic satisfy the Blaschke condition.

\subsection{Proof of the implication $(iii)\Rightarrow (ii)$}\label{17oct01}
\begin{theorem}\label{17oct02}
Assume that condition (iii) holds. That is, we assume that \eqref{17oct03} holds
and that (B) of Theorem \ref{mth-15oct01} $($equivalently $(b_1)$ of\eqref{27dec04}$)$ holds.
Then $m'(z)$ is of bounded characteristic, that is, it is a ratio of two bounded analytic functions.
\end{theorem}
\begin{proof}
Let $B_k$ be the Blaschke product over the orbit of $c_{k}$
\begin{equation*}%\label{17oct04}
B_k(z)=\prod\limits_{\gamma\in\Gamma}\frac{\gamma (z)-c_{k}}{\gamma (z)-\overline{c_{k}}}\ d_\gamma,\quad z\in\mathbb C_+,
\end{equation*}
where $|d_\gamma|=1$ are chosen so that the factors in $B_k$ are positive at $z_*$.
It converges since $\Gamma$ is of convergent type.
We now consider
\begin{equation*}%\label{17oct05}
B(z)= \prod\limits_{k}B_k(z).
\end{equation*}
This product converges due to the assumption \eqref{17oct03}. Moreover, it converges uniformly on compact subsets of $\mathbb C_+$.

Our goal is to prove that $\frac{B(z)}{m'(z)}$ is a bounded analytic function on $\mathbb C_+$.
More precisely, that
\begin{equation}\label{17oct06}
\left|\frac{B(z)}{m'(z)}\right|\le 1,\quad z\in\mathbb C_+.
\end{equation}
It turns out that it is easier to prove even stronger inequality
\begin{equation}\label{17oct07}
\left|\frac{B(z)}{m'(z)}\right|
\sum\limits_{\gamma\in\Gamma}
|\gamma'(z)|
\left(
1+\dfrac{1}{|\gamma (z)|^2}
\right)
\le 1,
\quad z\in\mathbb C_+.
\end{equation}
Easier because of the automorphic property of the latter function. Recall here that the series in \eqref{17oct07}
converges to a function continuous on $\mathbb C_+$, due to the assumption (B) of Theorem \ref{mth-15oct01} (equivalently ($b_1$) of \eqref{27dec04} ).
In other words we will prove that
\begin{equation}\label{17oct08}
\sum\limits_{\gamma\in\Gamma}
\left|
\dfrac{B(z)\gamma'(z)}{m'(z)}
\right|
+
\left|
\dfrac{B(z)\gamma'(z)}{m'(z)\gamma^2(z)}
\right|
\le 1,
\quad z\in\mathbb C_+.
\end{equation}
Observe that
$$
\dfrac{B(z)\gamma'(z)}{m'(z)}
\quad\text{and}\quad
\dfrac{B(z)\gamma'(z)}{m'(z)\gamma^2(z)}
$$
are holomorphic on $\mathbb C_+$. Therefore, their absolute values are
subharmonic functions on $\mathbb C_+$. Hence, the sum in \eqref{17oct08}
is a subharmonic function. Also the sum is automorphic with respect to
$\Gamma$.

We consider first the finitely generated approximation described in Section \ref{FZ}.
Recall that
$$
\L_n:\bbC_+/\G_n\simeq\O_n.
$$
Let $c_{k}^{(n)}$ be the
zero of $m'_n(z)$ on the $k$-th semicircle.
Let $B_k^{(n)}$ be the Blaschke product over the orbit of $c_{k}^{(n)}$ under $\Gamma_n$
$$
B_k^{(n)}(z)=\prod\limits_{\gamma\in\Gamma_n}\frac{\gamma(z)-c_{k}^{(n)}}{\gamma(z)-\overline{c_{k}^{(n)}}}\ d_\gamma,\quad z\in\mathbb C_+,
$$
if $k$-th semicircle is a part of the boundary of $\mathcal F_n$, and $B_k^{(n)}(z)=1$ otherwise.
We now consider
$$
B^{(n)}(z)= \prod\limits_{k}B_k^{(n)}(z).
$$
We are going to prove this approximative version of \eqref{17oct08}
\begin{equation}\label{17oct09}
\ff_n(z):=\sum\limits_{\gamma\in\Gamma_n}
\left|
\dfrac{B^{(n)}(z)\gamma'(z)}{m'_n(z)}
\right|
+
\left|
\dfrac{B^{(n)}(z)\gamma'(z)}{m'_n(z)\gamma^2(z)}
\right|
\le 1,
\quad z\in\mathbb C_+.
\end{equation}
Advantage of the function in \eqref{17oct09} over the function in \eqref{17oct08} is that the series in \eqref{17oct09}
converges in $\mathcal F_n$ and also on the boundary of $\mathcal F_n$ to a function continuous on $\mathcal F_n$ and
up to the boundary of $\mathcal F_n$ (including infinity),  since $\Gamma_n$ is finitely generated. The same is true for the fundamental domain of $\Gamma_n$, which is the union of $\mathcal F_n$ and the reflection of $\mathcal F_n$ about the $0$-th semicircle.

Similar to what we did in Section \ref{PT}, we define a subharmonic function $\fF_n(\l)$, $\l\in\O_n$, by
$$
\ff_n(z)=\fF_n(\L_n(z)).
$$
The function $\fF_n(\l)$ is continuous in $\O_n=\mathbb C\setminus E_n$ and also up to $E_n$.
By subharmonicity,  it attains its maximum on the boundary of $\O_n$. Thus, the maximum of $\ff_n(z)$ is attained on the part of the boundary of the fundamental domain that lies on the real axis. Recall that on the boundary of the fundamental domain all the series below converge to continuous functions. Therefore, for real $z$ on the boundary of the fundamental domain of $\Gamma_n$ we have,
by \eqref{26sep05}, \eqref{26sep04}, \eqref{25aug01} and \eqref{12oct04}, that
\begin{equation*}%\label{17oct10}
\frac{1}{|m_n'(z)|}
\sum\limits_{\gamma\in\Gamma_n}
|B^{(n)}(z)\gamma'(z)|
+
\left|
\dfrac{B^{(n)}(z)\gamma'(z)}{\gamma^2(z)}
\right|
= 1.
\end{equation*}
Here we used the fact that $\gamma(z)$ is real for real $z$ and that $\gamma'(z)$ is positive for real $z$.
Hence, \eqref{17oct09} follows, which is the approximative version of \eqref{17oct08}.

Now we want to pass to the limit in \eqref{17oct09} for arbitrary fixed $z\in\mathbb C_+$ as $n$ goes to infinity.
By Lemma \ref{L:29aug01}, $m'_n(z)$ converges to $m'(z)$.
The sum over $\Gamma_n$
converges to the sum over $\Gamma$.
It remains to show that $|B^{(n)}(z)|$ converges to $|B(z)|$.
Note that $|B^{(n)}_k(z)|=|g_n(c_{k}^{(n)}, z)|$ converges to
$|g(c_{k}, z)|=|B_k(z)|$, by Lemmas \ref{L:28aug01} and \ref{L:29aug01}. % and Remark \ref{R:28aug01}.
Further, $k\ne 0$,
$$
|B^{(n)}_k(z_*)|=|g_n(c_{k}^{(n)}, z_*)|\ge |g(c_{k}^{(n)}, z_*)| \ge |g(c_{k}, z_*)|.
$$
By assumption \eqref{17oct03} the product
$$
\prod\limits_{k\ne 0}|g(c_{k}, z_*)|
$$
converges (that is, it is greater than $0$). Then, by the Dominated Convergence
theorem\footnote{This case reduces to the standard Dominated Convergence by applying $(-\log)$ to the products.},
$$
\lim_{n\to\infty}|B^{(n)}(z_*)|=\lim_{n\to\infty}\prod\limits_{k\ne 0}|B^{(n)}_k(z_*)|
=\prod\limits_{k\ne 0}\lim_{n\to\infty}|B^{(n)}_k(z_*)|
$$
$$
=
\prod\limits_{k\ne 0}|B_k(z_*)|=|B(z_*)|.
$$
There exists a subsequence $n_j$ such that $B^{(n_j)}(z)$
converges for all $z\in\mathbb C_+$. Let
$$
\widetilde B(z)=\lim_{j\to\infty}B^{(n_j)}(z).
$$
Fix any $z\in\mathbb C_+$.
Then by Fatou's lemma\footnote{Same explanation as in the previous footnote.},
\begin{align*}
|\widetilde B(z)|&=\lim_{j\to\infty}|B^{(n_j)}(z)|=\lim_{j\to\infty}\prod\limits_{k}|B^{(n_j)}_k(z)|
\\
&\le\prod\limits_{k}\lim_{j\to\infty}|B^{(n_j)}_k(z)|
=\prod\limits_{k}|B_k(z)|=|B(z)|.
\end{align*}
Thus
$$
|\widetilde B(z)|\le |B(z)|,\quad z\in\mathbb C_+.
$$
Since
$$
|\widetilde B(z_*)|= |B(z_*)|,
$$
the equality must hold
\begin{equation}\label{24oct02}
|\widetilde B(z)|= |B(z)|,\quad z\in\mathbb C_+.
\end{equation}
Thus we get \eqref{17oct07} and, therefore, \eqref{17oct06}.
Since \eqref{24oct02} holds for every subsequential limit $\widetilde B(z)$ of $B^{(n)}(z)$,
we get
$$
 B(z)=\lim_{n\to\infty}B^{(n)}(z).
$$
\end{proof}
\if{%%%%%%%%%%%%%%%%%%%%%%%%%%%%%%%%%%%%%%%%%%%%%%%%%%%
\begin{remark}
Actually we proved that
\begin{equation}\label{31aug01}
\left|
\frac{B(z)}{g'(z, z_*)}
\sum\limits_{\gamma\in\Gamma}
\dfrac{\gamma'(z)}{(\gamma(z)-\overline{z_*})^2}
\right|
\le
\left|
\frac{B(z)}{g'(z, z_*)}
\right|
\sum\limits_{\gamma\in\Gamma}
\dfrac{|\gamma'(z)|}{|\gamma(z)-\overline{z_*}|^2}
\le 1,
\quad z\in\mathbb C_+.
\end{equation}
\end{remark}
}\fi%%%%%%%%%%%%%%%%%%%%%%%%%%%%%%%%%%%%%%%%%%%%%%%%%%
\begin{corollary}\label{17oct11}
$m'(z)$ is of bounded characteristic as the ratio of the following two bounded analytic functions
$$
B(z)\quad\text{and}\quad \frac{B(z)}{m'(z)}.
$$
\end{corollary}
\begin{remark}\label{17oct12}
Since function ${B(z)}/{m'(z)}$
%\begin{equation*}%\label{17oct13}
%\frac{B(z)}{m'(z)}
%\end{equation*}
is bounded, it can be written as
$$
\frac{B(z)}{m'(z)}=I(z)\cdot O(z),
$$
where $I(z)$ is an inner function and $O$ is a bounded outer function.
Moreover, $I(z)$ is a singular inner function, since the left hand side does not have zeros in $\mathbb C_+$.
Therefore,
\begin{equation}\label{17oct25}
m'(z)=\frac{B(z)}{O(z)I(z)}.
\end{equation}
\end{remark}
\begin{theorem} \label{17oct14} Function
$
{B(z)}/{m'(z)}
$
is outer. That is, $I(z)=1$.
\end{theorem}
The following facts are used to prove Theorem \ref{17oct14}.
\begin{lemma}\label{17oct15} $($Corollary \ref{19oct10} of Appendix$)$.
Let $x\in\mathbb R$. Then
a finite nontangential limits $m(x)$ and $m'(x)$ exist, $m(x)$ is real,
if and only if
$$
\sum\limits_{\gamma\in\Gamma}\gamma'(x)+
\frac{\gamma'(x)}{\gamma^2 (x)}<\infty .
$$
In this case
\begin{equation}\label{17oct16}
m'(x)=\sum\limits_{\gamma\in\Gamma}
\gamma'(x)
\left(
1+\dfrac{1}{\gamma^2 (x)}
\right).
\end{equation}
Hence, in our case ($m$ is a pure point and $m'$ is of bounded characteristic) \eqref{17oct16} holds almost everywhere on $\mathbb R$.
\end{lemma}
\begin{lemma}\label{17oct17} For every $z\in\mathbb C_+$ the following inequality holds
\begin{equation}\label{21oct01}
\frac 1 \pi \int\limits_{\mathbb R}\log
\sum\limits_{\gamma\in\Gamma}
\left(
\gamma'(x)+
\frac{\gamma'(x)}{\gamma (x)^2}
\right)
\frac{\Im z}{|x-z|^2}dx
\ge
\log
\sum\limits_{\gamma\in\Gamma}
\left(
|\gamma'(z)|+
\frac{|\gamma'(z)|}{|\gamma (z)|^2}
\right).
\end{equation}
\end{lemma}
\begin{proof}
Since
$$
\gamma'(z)=\frac{1}{(\gamma^{21}z+\gamma^{22})^2},
$$
one can write
$$
\sum\limits_{\gamma\in\Gamma}
|\gamma'(z)|
\left(
1+\dfrac{1}{|\gamma (z)|^2}
\right)
=
\sum\limits_{\gamma\in\Gamma}\phi_{\gamma}(z)^*\phi_\gamma(z),
$$
where
$$
\phi_{\gamma}(z)=
\begin{bmatrix}
\dfrac{1}{\gamma^{21}z+\gamma^{22}}
\\ \\
\dfrac{1}{\gamma^{21}z+\gamma^{22}}\cdot\dfrac{1}{\gamma(z)}
\end{bmatrix}.
$$
We consider functions
\begin{equation*}%\label{17oct18}
u_n(z)
=\sum\limits_{k=1}^n
|\gamma'_k(z)|
\left(
1+\dfrac{1}{|\gamma_k (z)|^2}
\right)
=
\sum\limits_{k=1}^n\phi_{\gamma_k}(z)^*\phi_{\gamma_k}(z)
,
\quad  \Im z>0.
\end{equation*}
From here we see that $u_n$ is a subharmonic function since
\begin{equation*}%\label{17oct19}
\frac{\partial^2}{\partial z \partial \overline z}u_n(z)
= \sum\limits_{k=1}^n\phi'_{\gamma_k}(z)^*\phi'_{\gamma_k}(z))\ge 0.
\end{equation*}
Also $\log u_n(z)$ is subharmonic, since
\begin{equation*}%\label{17oct20}
\frac{\partial^2}{\partial z \partial \overline z}\log u_n(z)
= -\frac{1}{u_n^2(z)}\frac{\partial u_n}{\partial z}
\frac{\partial u_n}{\partial \overline z}+\frac{1}{u_n}
\frac{\partial^2 u_n}{\partial z \partial \overline z}=
\end{equation*}
$$
\frac{1}{u_n^2(z)}
\left\{
\sum\limits_{k=1}^n\phi_{\gamma_k}(z)^*\phi_{\gamma_k}(z)
\sum\limits_{k=1}^n\phi'_{\gamma_k}(z)^*\phi'_{\gamma_k}(z)
-
\sum\limits_{k=1}^n\phi_{\gamma_k}(z)^*\phi'_{\gamma_k}(z)
\sum\limits_{k=1}^n\phi'_{\gamma_k}(z)^*\phi_{\gamma_k}(z)
\right\},
$$
which is nonnegative by Cauchy-Schwarz inequality. Therefore,
$$
\frac 1 \pi \int\limits_{\mathbb R}\log
\sum\limits_{k=1}^n
\left(
{\gamma_k}'(x)+
\frac{{\gamma_k}'(x)}{{\gamma_k} (x)^2}
\right)
\frac{\Im z}{|x-z|^2}dx
\ge
\log
\sum\limits_{k=1}^n
\left(
|{\gamma_k}'(z)|+
\frac{|{\gamma_k}'(z)|}{|{\gamma_k} (z)|^2}
\right).
$$
We now pass to the limit in this inequality. Since all integrands are nonnegative, the Monotone Convergence Theorem applies and we get \eqref{21oct01}.
\end{proof}
\begin{proof}[Proof of Theorem \ref{17oct14}]
By \eqref{26sep05}, for $z\in\mathbb C_+$
%\begin{equation}\label{17oct21}
%m'(z)=\sum\limits_{\gamma\in\Gamma}\gamma'(z)+
%\frac{\gamma'(z)}{\gamma^2 (z)}.
we have
\begin{equation}\label{17oct22}
|m'(z)|\le\sum\limits_{\gamma\in\Gamma}|\gamma'(z)|+
\frac{|\gamma'(z)|}{|\gamma (z)|^2}.
\end{equation}
Now, by Lemmas \ref{17oct15} and \ref{17oct17},
$$
\frac 1 \pi\int\limits_{\mathbb R}\log m'(x)\frac{\Im z}{|x-z|^2}dx
=
\frac 1 \pi \int\limits_{\mathbb R}\log
\sum\limits_{\gamma\in\Gamma}
\left(
\gamma'(x)+
\frac{\gamma'(x)}{\gamma (x)^2}
\right)
\frac{\Im z}{|x-z|^2}dx
$$
\begin{equation*}%\label{17oct23}
\ge
\log \sum_{\g\in\G}
\left(
|\gamma'(z)|+
\frac{|\gamma'(z)|}{|\gamma (z)|^2}
\right)
.
\end{equation*}
On the other hand (see \eqref{17oct25})
\begin{equation*}%\label{17oct26}
\frac 1 \pi\int\limits_{\mathbb R}\log m'(x)\frac{\Im z}{|x-z|^2}dx
=
-\frac 1 \pi\int\limits_{\mathbb R}\log|O(x)|\frac{\Im z}{|x-z|^2}dx = -\log|O(z)|,
\end{equation*}
since $O$ is a bounded outer function. Thus,
\begin{equation}\label{17oct27}
\sum_{\g\in\G}|\gamma'(z)|+
\frac{|\gamma'(z)|}{|\gamma (z)|^2}
\le
\frac{1}{|O(z)|}
,\quad z\in \mathbb C_+.
\end{equation}
Combining \eqref{17oct27} with \eqref{17oct22} and \eqref{17oct25}, we get
$$
\left|\frac{B(z)}{O(z)I(z)}\right|=
|m'(z)|\le \frac{1}{|O(z)|},\quad z\in \mathbb C_+.
$$
That is,
$$
\left|\frac{B(z)}{I(z)}\right|
\le 1.
$$
The latter implies that $I(z)=1$.
\end{proof}

\subsection{Proof of the implication $(i)\Rightarrow (iii)$}

Let $\cH^2(\a)$ be non trivial for all $\a\in\G^*$. Then $H^2(\a)$ is non trivial for all $\a$, that is, the Widom condition holds
$$
\sum_{\mu:\nabla G(\mu,\l_*)=0} G(\mu,\l_*)<\infty.
$$
This in turn implies (A) of condition (iii) in Theorem \ref{mth-15oct01}, since if $\mu_{\l_*}$ is the critical point of $G(\l,\l_*)$ in the $k$-th gap and
$\mu$ is any point in this gap, then
$$
G(\mu,\l_*)\le G(\mu_{\l_*},\l_*).
$$
Now, since the Widom condition holds, $\G$ acts on $\bbR$ dissipatively, that is, there exists a measurable fundamental set $\bbE\subset \bbR$. Let $f$ be a non trivial function from  $\cH^2(\alpha)$.  Then
\begin{equation*}%\label{29aug01}
\int\limits_\bbR |f(x)|^2dx=
\sum\limits_{\gamma\in\Gamma}\
\int\limits_\mathbb E |f(x)|^2\gamma'(x) dx
=
\int\limits_\mathbb E \sum\limits_{\gamma\in\Gamma}\ |f(x)|^2\gamma'(x) dx
,
\end{equation*}
the latter equality is due to Fubini's theorem. Therefore,
\begin{equation*}%\label{29aug02}
|f(x)|^2\sum\limits_{\gamma\in\Gamma}\gamma'(x) <\infty
\end{equation*}
 almost everywhere on $\mathbb E$.
Since $f\neq 0$ (almost everywhere)
  we have
\begin{equation*}%\label{29aug03}
\sum\limits_{\gamma\in\Gamma}\gamma'(x)<\infty\quad \text{for a.e.}\ x\in\bbE.
\end{equation*}
We fix one such $x$, then for $z_0=x+i$ we obtain
$$
\sum\limits_{\gamma\in\Gamma}|\gamma'(z_0)|<\infty.
$$
 By the Harnack inequality we have \eqref{23oct1}. By Proposition \ref{pr23oct},  $(b)$ holds and it is equivalent to
 condition (B) of Theorem \ref{mth-15oct01}.

\subsection{Proof of the implication $(ii)\Rightarrow (i)$}

\begin{lemma}\label{29dec01}
If (a) and (b) hold, then $H^2(\a)$ are non trivial for all $\a\in\G^*$.
\end{lemma}
\begin{proof}
By %% Theorem \ref{17oct02} and
Lemma \ref{17oct15}, we have that under assumptions (a) and (b)
$$
m'(x)=\sum\limits_{\gamma\in\Gamma}
\gamma'(x)
\left(
1+\dfrac{1}{\gamma^2 (x)}
\right)
$$
almost everywhere on $\mathbb R$ and that
\begin{equation*}%\label{9aug2}
0<\int\limits_{\mathbb R}\frac{\log m'(x)}{1+x^2} dx<\infty.
\end{equation*}
Then
\begin{equation*}%\label{9aug1}
\rho(x)=\sum_{\g\in\G}{\g'(x)}.
\end{equation*}
also converges almost everywhere on $\mathbb R$ and
\begin{equation}\label{9aug2}
0<\int\limits_{\mathbb R}\frac{\log \rho(x) }{1+x^2}dx<\infty.
\end{equation}
Consider the following function on $\mathbb R$ (compare to \eqref{16sep02}, \eqref{16sep03}, also to \eqref{30aug05})
%% density generated by $dm$, see \eqref{24oct1},
$$
\rho_{i}(x)=\sum_{\g\in\G}\frac{\g'(x)}{1+\g(x)^2}.
$$
Since
$$
\frac 1{1+x^2}\le \rho_{i}(x)\le \rho(x),
$$
we have
\begin{equation}\label{9aug3}
-\infty<\int\frac{\log \rho_{i}(x) dx}{1+x^2}<\infty .
\end{equation}
Combining inequality \eqref{21oct02} of Lemma \ref{02sep06} with inequality \eqref{02sep05}, we conclude that
$\log^+|g'(z,i)|$ has a harmonic majorant in the upper half plane. This means that
$g'(z, i)$ is of bounded type on the upper half plane. Therefore, by Theorem \ref{TP},
all $H^2(\a)$ are non-trivial.
\end{proof}
\noindent
Inequalities \eqref{9aug2} and \eqref{9aug3} allow to define
an outer function $\phi(z)$ by
\begin{equation}\label{9aug4}
|\phi(x)|^2=\frac{\rho_{i}(x)}{\rho(x)}\le 1.
\end{equation}
We denote by $\a_\phi$ the character associated to this function.
\begin{proposition}\label{pr25}
If (a) and (b) hold, then $\cH^2(\a)=\phi H^2(\a_{\phi}^{-1}\a)$.
\end{proposition}
\begin{proof}
We first show that $\cH^2(\a)\subseteq\phi H^2(\a_{\phi}^{-1}\a)$.
If $f\in \cH^2(\a)$, then $h=f/\phi$ is of Smirnov class. Recall that $\bbE$ is the fundamental measurable set for the action of $\G$ on $\bbR$.
In view of \eqref{9aug4}, we have
 $$
 \int_\bbR |h|^2 \frac{dx}{1+x^2}=\int_{\bbE}|h(x)|^2\rho_{i}(x)dx=\int_{\bbE}|f(x)|^2\rho(x)dx=\int_\bbR |f|^2 dx.
 $$
 Then, by the Smirnov maximum principle, $h\in H^2$, and, therefore, $h\in H^2(\a_\phi^{-1}\a)$. The converse inclusion is proved the same way.
\end{proof}
\begin{corollary}
If (a) and (b) hold, then $\cH^2(\a)$ are non trivial for all $\a\in\G^*$.
\end{corollary}
\begin{proof}
This is a straightforward combination of Lemma \ref{29dec01} and
Proposition \ref{pr25}.
\end{proof}
\begin{remark}
We point out that an analogue of the space $\cH^2(\a)$ still can be defined
 for Widom domains (see the definition below)
even if condition (b) is violated. Indeed, the density outer function (compare to \eqref{9aug4})
\begin{equation*}%\label{23oct9}
|\Phi(\l)|^2=\frac{\t_{\l_*}'(\l)}{\t'(\l)},\quad \l\in E,
\end{equation*}
where $\t_{\l_*}$ and $\t$ are defined in \eqref{6aug1} and \eqref{11sep7},
is always well defined in Widom domains  due to Theorem D \cite{SY}. This suggests the following
\begin{definition}\label{defoct23} Let $\Omega$ be of Widom type. Let $\pi_1(\O)\simeq\G$ be the fundamental group of this domain.
For a character $\a\in\pi_{1}(\O)^*$
we say that a function $F$ belongs to $\cH^2_{\O}(\a)$ if it is a character-automorphic  multivalued function in the domain, i.e.,
$$
F(\tilde\g(\l))=\a(\tilde\g) F(\l), \quad \tilde \g\in\pi_1(\O),
$$
and $|F(\l)/\Phi(\l)|^2$ possesses a harmonic majorant in $\Omega$.
\end{definition}

\end{remark}

%%%%%%%%%%
%%%%%%%%%%%
\section{Appendix: Carath\'eodory and Frostman theorems}
Theorems of Carath\'eodory and Frostmant that are used in the proofs of Pommerenke theorem
(Theorem \ref{Pommerenke})  and in the most important part (Theorem \ref{17oct02}) of our main theorem  depend on the following
theorem due to Carath\'eodory and Julia \cite{Car}, for a modern exposition see, e. g., \cite{BoKh} and further references there.
\begin{theorem}[Carath\'eodory--Julia, \cite{Car}]\label{03sep01}
Let function $w$ be analytic in the unit disk and bounded in modulus by $1$.
Let $t_0$ be a point on the unit circle. The following are equivalent:
\begin{eqnarray}
&&(1) \quad d_1:={\displaystyle\liminf_{z\to
t_0}\frac{1-|w(z)|^2}{1-|z|^2}}<\infty\quad (|z|<1, z {\text\ approaches\ } t_0
{\text\ in\ an\ arbitrary\ way });\nonumber \\
&&(2) \quad
d_2:={\displaystyle\lim_{z\to t_0}
\frac{1-|w(z)|^2}{1-|z|^2}}<\infty
\quad (z {\text\ approaches\ } t_0
{\text\ nontangentially });\nonumber\\
&&(3)\quad \mbox{Finite nontangential limits} \; \;  \nonumber\\
&&\qquad\qquad w(t_0):={\displaystyle\lim_{z\to
t_0}w(z)} \; \; \mbox{and} \; \;
d_3:={\displaystyle\lim_{z\to
t_0}\frac{1-w(z)\overline{w(t_0)}}{1-z\bar{t}_0}}\nonumber\\
&&\qquad\ \mbox{exist},\ |w(t_0)|=1.\nonumber\\
&&(4)\quad \mbox{Finite nontangential limits} \; \;  \nonumber\\
&&\qquad\qquad w(t_0):={\displaystyle\lim_{z\to
t_0}w(z)} \; \; \mbox{and} \; \;
w'(t_0)={\displaystyle\lim_{z\to
t_0}\frac{w(z)-w(t_0)}{z-t_0}}\nonumber\\
&&\qquad\ \mbox{exist}, |w(t_0)|=1.\ w'(t_0)\ \text{is called the angular derivative at}\ t_0.\nonumber\\
&&(5)\quad \mbox{Finite nontangential limits} \; \; \nonumber\\
&&\qquad\qquad w(t_0):={\displaystyle\lim_{z\to
t_0}w(z)} \; \; \mbox{and} \; \; w'_0:={\displaystyle\lim_{z\to
t_0}w^\prime(z)}\nonumber\\
&&\qquad\ \mbox{exist},\ |w(t_0)|=1.\nonumber\\
&&(6)\quad \mbox{There exist a constant $w_0$, $|w_0|=1$ and a constant $d\ge 0$
} \; \; \nonumber\\
&&\qquad \mbox{\ such that the boundary Schwarz-Pick inequality holds}\nonumber\\
&&\qquad\qquad
\left|
\frac{w(z)-w_0}{z-t_0}
\right|^2\le d\cdot \frac{1-|w(z)|^2}{1-|z|^2},\quad |z|<1;
\label{28oct01}\\
&&\qquad \mbox{inequality \eqref{28oct01} implies that the following nontangential limit} \; \; \nonumber\\
&&\qquad\qquad w(t_0):={\displaystyle\lim_{z\to
t_0}w(z)} \; \; \mbox{exists\ and} \; \; w(t_0)=w_0;\nonumber\\
&&\qquad \mbox{\ we denote the smallest constant $d$ that works for \eqref{28oct01}  by $d_4$.}\nonumber
\end{eqnarray}
When these conditions hold, we have $w'_0=w'(t_0)$ and
$$d_1=d_2=d_3=d_4=t_0\frac{w'(t_0)}{w(t_0)}=|w'(t_0)|.$$
This number is equal to $0$ if and only if $w$ is a unimodular constant.
\end{theorem}
%%%%%%%%% (5) implies (4), (4) implies (1) %%%%%%%%%%%%%
\begin{theorem}[Carath\'eodory, \cite{Car}] \label{03sep08}
Let $w$, $w_n$ be analytic functions bounded in modulus by $1$ on the unit disk.
Assume that $w_n(z)$ converges to $w(z)$ for every $|z|<1$.
Let $|t_0|=1$.
Assume that nontangential boundary values $w_n(t_0)$,
$w'_n(t_0)$ exist and that $|w_n(t_0)|=1$, $w'_n(t_0)$
are finite.
We assume that $$\underline{\lim}\ |w'_n(t_0)|<\infty.$$
Then
the nontangential boundary values $w(t_0)$, $w'(t_0)$
exist, $|w(t_0)|=1$ and
%% $$
%% w(t_0)=\lim w_n(t_0)
%% $$
%% and
\begin{equation}\label{28oct03}
|w'(t_0)|\le\underline{\lim}\ |w'_n(t_0)|.
\end{equation}
\end{theorem}
\begin{proof}
%%%%%%%%If $\underline{\lim}\ |w'_n(t_0)|=\infty$, then inequality \eqref{28oct03} is obvious. So, we may assume that
%%%%%%%%$\underline{\lim}\ |w'_n(t_0)|<\infty$.
Let $|w'_{n_k}(t_0)|$ converge to $\underline{\lim}\ |w'_n(t_0)|$.
By \eqref{28oct01}, we have
$$
\left|
\frac{w_{n_k}(z)-w_{n_k}(t_0)}{z-t_0}
\right|^2\le |w'_{n_k}(t_0)|\cdot \frac{1-|w_{n_k}(z)|^2}{1-|z|^2},\quad |z|<1.
$$
%%%%%%%%%%%%%%%%In particular, this holds for the subsequence $n_k$.
$w_{n_k}(t_0)$ is a sequence of complex numbers of modulus one. Therefore, there exists a convergent subsequence $w_{n_{k_j}}(t_0)$. We denote the limit by $w_0$, $|w_0|=1$. Since $w_{n_{k_j}}(z)$ converge to $w(z)$ for every $|z|<1$, we get (by passing to the limit as $j\to\infty$)
$$
\left|
\frac{w(z)-w_0}{z-t_0}
\right|^2\le \underline{\lim}\ |w'_{n}(t_0)|\cdot \frac{1-|w(z)|^2}{1-|z|^2},\quad |z|<1.
$$
From here we see that $w_0=w(t_0)$
%%%%%%%Thus, sequence $w_n(t_0)$ has only one subsequential limit $w(t_0)$. Therefore,
%%%%%%%$$
%%%%%%%\lim w_n(t_0)=w(t_0)
%%%%%%%$$
and we get
$$
\left|
\frac{w(z)-w(t_0)}{z-t_0}
\right|^2\le \underline{\lim}\ |w'_{n}(t_0)|\cdot \frac{1-|w(z)|^2}{1-|z|^2},\quad |z|<1.
$$
By Theorem \ref{03sep01}, the latter inequality implies that $w'(t_0)$ exists and that it is finite.
Since the smallest constant that works for this inequality is $|w'(t_0)|$, \eqref{28oct03} follows.

\end{proof}
\begin{theorem}[Frostman, \cite{Fro}]\label{19oct01}
In addition to assumptions of Theorem \ref{03sep08},
assume that $|w_n(z)|\ge |w(z)|$ for every $z$, $|z|<1$.
Then
$$
w(t_0)=\lim w_n(t_0)\quad \text{and}\quad w'(t_0)=\lim w'_n(t_0).
$$
\end{theorem}
\begin{proof}
By assumption,
$$
\frac{1-|w_n(z)|^2}{1-|z|^2}\le \frac{1-|w(z)|^2}{1-|z|^2}.
$$
Therefore,
$$
\lim_{z\to t_0} \frac{1-|w_n(z)|^2}{1-|z|^2}\le
\lim_{z\to t_0} \frac{1-|w(z)|^2}{1-|z|^2}.
$$
That is, in view of Theorem \ref{03sep01},
$$
|w'_n(t_0)|\le |w'(t_0)|.
$$
Hence, we get
$$
\overline{\lim} |w'_n(t_0)|\le |w'(t_0)|.
$$
Combining this with Theorem \ref{03sep08}, we get
$$
|w'(t_0)|=\lim |w'_n(t_0)|.
$$
The first assertion of the theorem follows from the observation that now one does not need to choose a subsequence at the beginning of the proof of Theorem \ref{03sep08}.
This implies that every subsequential limit of $w_n(t_0)$ is $w(t_0)$. After that, the second assertion is a consequence of the relation
$$
|w'(t_0)|=t_0\frac{w'(t_0)}{w(t_0)}.
$$
\end{proof}
By a simple  substitution
$$
z:=\frac{z-i}{z+i},
$$
that maps upper half plane onto the unit disk,
Theorems \ref{03sep01} and \ref{19oct01} can be restated for functions on the upper half plane.
\begin{theorem}[Carath\'eodory--Julia]\label{21oct06}
Let $w$ be analytic on the upper half plane and bounded in modulus by $1$.
Let $x\in\mathbb R$ be a point on the real axis. The following are equivalent:
\begin{eqnarray}
&&(1) \quad d_1:={\displaystyle\liminf_{z\to
x}\frac{1-|w(z)|^2}{2\Im z}}<\infty\quad (\Im z>0,\ z {\text\ approaches\ } x
{\text\ in\ an\ arbitrary\ way }); \nonumber\\
&&(2) \quad
d_2:={\displaystyle\lim_{z\to x}
\frac{1-|w(z)|^2}{2\Im z}}<\infty
\quad (\Im z>0,\ z {\text\ approaches\ } x
{\text\ nontangentially });\nonumber\\
&&(3)\quad \mbox{Finite nontangential limits} \; \;  \nonumber\\
&&\qquad\qquad w(x):={\displaystyle\lim_{z\to
x}w(z)} \; \; \mbox{and} \; \;
d_3:={\displaystyle\lim_{z\to
x}\frac{1-w(z)\overline{w(x)}}{i(x-z)}}\nonumber\\
&&\qquad\ \mbox{exist}, |w(x)|=1.\nonumber\\
&&(4)\quad \mbox{Finite nontangential limits} \; \;  \nonumber\\
&&\qquad\qquad w(x):={\displaystyle\lim_{z\to
x}w(z)} \; \; \mbox{and} \; \;
w'(x)={\displaystyle\lim_{z\to
x}\frac{w(z)-w(x)}{z-x}}\nonumber\\
&&\qquad\ \mbox{exist}, |w(x)|=1.\nonumber\\
&&(5)\quad \mbox{Finite nontangential limits} \; \; \nonumber\\
&&\qquad\qquad w(x):={\displaystyle\lim_{z\to
x}w(z)} \; \; \mbox{and} \; \; w'_0:={\displaystyle\lim_{z\to
x}w^\prime(z)}\nonumber\\
&&\qquad\ \mbox{exist}, |w(x)|=1.
w'(x)\ \text{is caled the angular derivative at}\ x.\nonumber\\
&&(6)\quad \mbox{There exist a constant $w_0$, $|w_0|=1$ and a constant $d\ge 0$
} \; \; \nonumber\nonumber\\
&&\qquad \mbox{\ such that the boundary Schwarz-Pick inequality holds}\nonumber\nonumber\\
&&\qquad\qquad
\left|
\frac{w(z)-w_0}{z-x}
\right|^2\le d\cdot \frac{1-|w(z)|^2}{2\Im z},\quad \Im z>0;
\label{28oct04}\\
&&\qquad \mbox{inequality \eqref{28oct04} implies that the following nontangential limit} \; \; \nonumber\nonumber\\
&&\qquad\qquad w(x):={\displaystyle\lim_{z\to
x}w(z)} \; \; \mbox{exists\ and} \; \; w(x)=w_0;\nonumber\nonumber\\
&&\qquad \mbox{\ we denote the smallest constant that works for \eqref{28oct04}  by $d_4$.}\nonumber
\end{eqnarray}
When these conditions hold, we have $w'_0=w'(x)$ and
$$|w'(x)|=\frac{1}{i}\frac{w'(x)}{w(x)}=d_1=d_2=d_3=d_4.$$
\end{theorem}
The next theorem is a version of Theorem \ref{19oct01} for the upper half plane.
\begin{theorem}[Frostman, \cite{Fro}]\label{19oct11}
Let $w$, $w_n$ be analytic functions on the upper half plane bounded in modulus by $1$.
Assume that $w_n$ converge to $w$ for every $z\in\mathbb C_+$ and that
$|w_n(z)|\ge |w(z)|$ for every $z\in\mathbb C_+$.
Let $x\in\mathbb R$.
Let $w_n(x)$, and $w'_n(x)$ be the nontangential boundary values,
$|w_n(x)|=1$, $w'_n(x)$ is finite.
Assume that $$\underline{\lim}\ |w'_n(x)|<\infty.$$
Then nontangential boundary values $w(x)$ and $w'(x)$ exist,
$|w(x)|=1$, $w'(x)$ is finite and
$$
w(x)=\lim w_n(x),\quad w'(x)=\lim w'_n(x).
$$
\end{theorem}
\begin{corollary}[Frostman, \cite{Fro}]\label{19oct09}
Let $w$ be a Blaschke product on the upper half plane
$$
w(z)=\prod\limits_k B_k(z).
$$
Let $x\in\mathbb R$.
%%%%%% Assume that $w(x)$ exists and $|w(x)|=1$. Assume that $w'(x)$ exists and that it is finite. Then
Then $w(x)$ and $w'(x)$ exist with $|w(x)|=1$, $w'(x)$ finite
if and only if
$$
\sum\limits_k |B'_k(x)|<\infty.
$$
In this case
$$
|w'(x)|=\sum\limits_k |B'_k(x)|.
$$
\end{corollary}
\begin{proof}
Let $w_n$ be a finite Blaschke product
$$
w_n(z)=\prod\limits_{k=1}^n B_k(z).
$$
Then
$$
\frac{w'_n(z)}{w_n(z)}=\sum\limits_{k=1}^n \frac{B'_k(z)}{B_k(z)}.
$$
Observe that (compare to Theorem \ref{21oct06})
$$
|B_k(x)|=1,\ \text{and}\ \ \frac{1}{i}\cdot\frac{B'_k(x)}{B_k(x)}= |B'_k(x)|.
$$
Same is true for $w_n$ at point $x$. Therefore, we get
$$
|w'_n(x)|=\sum\limits_{k=1}^n |B'_k(x)|.
$$
Since $w_n$ is a divisor of $w$ the following inequality holds
$$
|w_n(z)|\ge|w(z)|
$$
for every $z\in\mathbb C_+$. Also $w_n(z)$ converge to $w(z)$ for every $z\in\mathbb C_+$.
If $|w(x)|=1$, $w'(x)$ exists and it is finite,
then (like in Theorem \ref{19oct01})
$$
|w_n'(x)|\le|w'(x)|.
$$
%%%  and we are in the situation of Theorem \ref{19oct11}.
Therefore, for every $n$
$$
\infty > |w'(x)|\ge |w'_n(x)|=\sum\limits_{k=1}^n |B'_k(x)|
$$
and
$$
\sum\limits_{k=1}^\infty |B'_k(x)|<\infty.
$$
%% From here we see that both sides exist or do not exist simultaneously. Also recall that $w'(t_0)=0$ if and only if
%% $w$ is a constant.

Conversely, if the latter sum converges,
then $|w_n'(x)|$ are bounded and we are in the situation of Theorem \ref{19oct11}.
\end{proof}
\begin{corollary}\label{19oct10}
Let $m(z)$ be the Martin function with a pure point measure, defined as in \eqref{26sep04}
\begin{equation*}%\label{19oct05}
m(z)=\sum\limits_{\gamma\in\Gamma}
\left(\gamma(z)-\frac{1}{\gamma(z)}\right)
-
\Re\left(
\gamma(i)
-\frac{1}{\gamma(i)}\right).
\end{equation*}
Recall that by \eqref{26sep05}
\begin{equation*}%\label{19oct06}
m'(z)=\sum\limits_{\gamma\in\Gamma}\gamma'(z)+
\frac{\gamma'(z)}{\gamma^2 (z)}
\end{equation*}
Let $x\in\mathbb R$. Then
a finite nontangential limits $m(x)$ and $m'(x)$ exist, $m(x)$ is real,
if and only if
$$
\sum\limits_{\gamma\in\Gamma}\gamma'(x)+
\frac{\gamma'(x)}{\gamma^2 (x)}<\infty .
$$
In this case
\begin{equation*}%\label{19oct07}
m'(x)=\sum\limits_{\gamma\in\Gamma}\gamma'(x)+
\frac{\gamma'(x)}{\gamma^2 (x)}.
\end{equation*}
\end{corollary}
\begin{proof}
Consider the following inner function
$$
w(z)=e^{i m(z)}.
$$
Observe that
$$
\frac{w'(z)}{w(z)}=i m'(z).
$$
Consider
\begin{equation*}%\label{19oct08}
m_n(z)=\sum\limits_{k=1}^n
\left(\gamma_k(z)-\frac{1}{\gamma_k(z)}\right)
-
\Re\left(
\gamma_k(i)
-\frac{1}{\gamma_k(i)}\right)
\end{equation*}
and the corresponding inner function
$$
w_n(z)=e^{i m_n(z)}.
$$
Then
$$
\frac{w'_n(z)}{w_n(z)}=i m'_n(z).
$$
In view of formula \eqref{26sep06}, $\Im m_n(z)$ increases in $n$ for every $\Im z>0$. Therefore,
$|w_n(z)|$ decreases in $n$. If finite nontangential limits $m(x)$ and $m'(x)$ exist,
$m(x)$, then finite nontangential limits $w(x)$ and $w'(x)$ exist, $|w(x)|=1$.Therefore,
we are in the situation of Theorem \ref{19oct11}. Hence,
$$
m'(x)=\frac{1}{i}\frac{w'(x)}{w(x)}=|w'(x)|=\lim |w'_n(x)|=\lim\frac{1}{i}\frac{w_n'(x)}{w_n(x)}
$$
$$
=\lim m'_n(x)=\lim\sum\limits_{k=1}^n\gamma_k'(x)+
\frac{\gamma_k'(x)}{\gamma_k^2 (x)}
=\sum\limits_{\gamma\in\Gamma}\gamma'(x)+
\frac{\gamma'(x)}{\gamma^2 (x)}.
$$
Conversely, if
$$
\sum\limits_{\gamma\in\Gamma}\gamma'(x)+
\frac{\gamma'(x)}{\gamma^2 (x)}<\infty ,
$$
then $|w'_n(x)|=m'_n(x)$ are bounded and we are again in the situation of Theorem \ref{19oct11}.
\end{proof}

\bigskip

A. Kheifets, Department of Mathematical Sciences, University of Massachusetts Lowell, One University Ave.,
Lowell, MA 01854,USA

\emph{E-mail address:} {Alexander\underline{ }Kheifets@uml.edu}

\medskip

P. Yuditskii, Abteilung f\"ur Dynamische Systeme und Approximationstheorie, Institut f\"ur Analysis, Johannes Kepler Universit\"at Linz, A-4040 Linz, Austria

\emph{E-mail address:} {Petro.Yudytskiy@jku.at}


\begin{thebibliography}{99}

\bibitem{AV}
D. Alpay, V. Vinnikov, \emph{Indefinite Hardy spaces on finite bordered Riemann surfaces}, J. Funct. Anal. 172 (2000), no. 1, 221--248.

\bibitem{BV}
J. Ball, V. Vinnikov, \emph{Zero-pole interpolation for matrix meromorphic functions on a compact Riemann surface and a matrix Fay trisecant identity}, Amer. J. Math. 121 (1999), no. 4, 841--888.

\bibitem{BoKh}
V. Bolotnikov, A. Kheifets,  \emph{A higher order analogue of the Caratheodory-Julia theorem}, Journal of Functional Analysis, 237, no. 1 (2006), 350 - 371

\bibitem{BS}{A. Borichev, M. Sodin, }\textit{Krein's entire functions and Bernstein approximation problem},
Illinois Journal of Mathematics,  45, no. 1 (2001), 167--185.

\bibitem{Car}{C. Carath\' eodory, }\textit{Theory of Functions of a Complex Variable},
Engl. Translation, Chelsea Publishing Company, NY, 1960.

\bibitem{EYu}{A. Eremenko, P. Yuditskii, }  {\em Comb functions.} Recent advances in orthogonal polynomials, special functions, and their applications, Contemp. Math.,
{\bf 578} (2012), 99--118.

\bibitem{Fro}
O. Frostman, \emph{Sur les produits de Blaschke}, Kungl. Fysiografiska S\"alskapets I Lund F\"orhandlingar,
12, no.15 (1943), 169 - 182.

%
\bibitem{Hasu} M. Hasumi,
\emph{Hardy Classes on Infintely Connected Riemann Surfaces},
LNM \textbf{1027}, Springer, New York, Berlin, 1983.

\bibitem{Kh}
A. Kheifets, \emph{Abstract interpolation problem and some applications}, Lecture notes given in framework of Holomorphic Spaces semester, fall 1995, in: Holomorphic Spaces (S. Axler, J. McCarthy, D. Sarason editors), MSRI Publications, 33 (1998) 351-381, Berkeley, California.

\bibitem{Koo}
P.~Koosis, \emph{The logarithmic integral. {I}}, Cambridge Studies in Advanced
  Mathematics, vol.~12, Cambridge University Press, Cambridge, 1998, Corrected
  reprint of the 1988 original.



\bibitem{Nik}
N.~K. Nikolskii, \emph{Treatise on the shift operator}, A Series of
  Comprehensive Studies in Mathematics, Spriger-Verlag, Berlin Heidelberg New
  York Tokyo, 1986.
  %\MR{0827223}

\bibitem{Lev}
B. Ya. Levin, \textit{Majorants in classes of subharmonic functions, II, The relation between majorants and conformal mapping, III, The classification of the closed sets on $\bbR$ and the representation of the majorants}, Teor. Funktsii Funktsional. Anal. i Prilozhen. 52 (1989), 3--33; English translation: J. Soviet Math. 52 (1990), 3351--3372.

\bibitem{LKMV}
M. S. Liv\v sic, N. Kravitsky, A. S. Markus, V. Vinnikov,
\textit{Theory of Commuting Nonselfadjoint Operators}, Mathematics and Its Applications, 332, Kluwer Academic, Dordrecht (1995)

\bibitem{Mar}
V.~Marchenko, \emph{Sturm-Liouville operators and applications}, Birkh\"auser Verlag, Basel, 1986.

\bibitem{PF}
B. S. Pavlov, S. I. Fedorov, \textit{Shift group and harmonic analysis on a Riemann surface of genus one}, Algebra i Analiz, 1:2 (1989), 132--168; Leningrad Math. J., 1:2 (1990), 447--490


\bibitem{Pom} Ch. Pommerenke, \textit{On the Green's function of Fuchsian groups,}
Ann. Acad. Sci. Fenn., 2 (1976), 409-427.



\bibitem{SY} M.~Sodin and P.~Yuditskii,
{\it Almost periodic Jacobi matrices with homogeneous spectrum, infinite-dimensional Jacobi
inversion, and Hardy spaces of character-automorphic functions},
J.\ Geom.\ Anal. {\bf 7} (1997), 387--435.



\bibitem{VYis}
A. Volberg and P. Yuditskii, \textit{
On the inverse scattering problem for Jacobi matrices with the spectrum on an interval, a finite system of Intervals or a Cantor set of positive length},
Communications in Mathematical Physics, {\bf 226} (2002), 567--605.

\bibitem{VY}
A. Volberg and P. Yuditskii, \textit{Kotani-Last problem and Hardy spaces on surfaces of Widom type}. Invent. Math., 197 (2014), No. 3, 683-740.

\bibitem{VYM}
A. Volberg and P. Yuditskii, \textit{Mean type of functions of bounded characteristic and Martin functions in Denjoy domains}, Adv. in Math.,  290 (2016), 860--887.

%
\bibitem{Widom} H.~Widom,
{\it Extremal polynomials associated with a system of curves in the complex plane},
Adv. in Math., {\bf 3} (1969), 127--232.

%
\bibitem{Widom71} H.~Widom,
{\it $H^p$ sections of vector bundles over Riemann surfaces},
Ann. of Math., {\bf 94} (1971), 304--324.



\bibitem{You}
J. You,
\textit{Quantitative almost reducibility and its applications}, Proc. Int. Cong. of Math. - 2018, Rio de Janeiro, Vol. 2, 2107--2128.


\end{thebibliography}
 \end{document}